\documentclass{amsart}

\title{On weighted Bergman spaces of a domain with Levi-flat boundary}

\author{Masanori Adachi}
\address{Department of Mathematics, Faculty of Science, Shizuoka University.  836 Ohya, Suruga-ku, Shizuoka 422-8529, Japan.}
\email{adachi.masanori@shizuoka.ac.jp}


\date{\today}

\usepackage{amsmath,amsthm,amssymb,latexsym}
\usepackage[abbrev]{amsrefs}
\usepackage{color}

\usepackage[all]{xy}

\newtheorem*{MainTheorem}{Main Theorem}
\newtheorem*{Theorem*}{Theorem}
\newtheorem*{Conjecture*}{Conjecture}
\newtheorem*{Claim}{Claim}

\newtheorem{Theorem}{Theorem}[section]
\newtheorem{Proposition}[Theorem]{Proposition}

\newtheorem{Lemma}[Theorem]{Lemma}
\newtheorem{Corollary}{Corollary}

\theoremstyle{remark}

\newcommand\C{\mathbb{C}}  

\newcommand\N{\mathbb{N}}

\newcommand\D{\mathbb{D}}
\newcommand\Beta{\mathrm{B}}
\newcommand\PP{\mathbb{P}}
\newcommand\Ker{\operatorname{Ker}}

\newcommand\Aut{\operatorname{Aut}}

\newcommand{\bd}{\partial}
\newcommand{\pa}{\partial}
\newcommand{\opa}{\overline\pa}
\newcommand{\ol}{\overline }

\begin{document}

\maketitle

\begin{abstract}
The aim of this study is to understand to what extent a 1-convex domain with Levi-flat boundary is capable of holomorphic functions with slow growth.
This paper discusses a typical example of such domain, the space of all the geodesic segments on a hyperbolic compact Riemann surface. 
Our main finding is an integral formula that produces holomorphic functions on the domain from holomorphic differentials on the Riemann surface.
This construction can be seen as a non-trivial example of $L^2$ jet extension of holomorphic functions with optimal constant.  
As its corollary, it is shown that the weighted Bergman spaces of the domain is infinite dimensional for any weight order greater than $-1$ in spite of the fact that the domain does not admit any non-constant bounded holomorphic functions. 
\end{abstract}

\section{Introduction}
Denote by $\D$ the unit disk in $\C$, and let $\Sigma = \D/\Gamma$ be a compact Riemann surface of genus $\geq 2$ where $\Gamma$ is a Fuchsian group $< \Aut(\D) = PSU(1,1)$. We consider a quotient of the bidisk by $\Gamma$, $\Omega := \D \times \D / \Gamma$, 
where $\Gamma$ acts on $\D \times \D$ diagonally, 
\[
\gamma \cdot (z, w) = (\gamma z, \gamma w)
\]
for each $\gamma \in \Gamma$ and $(z, w) \in \D \times \D$. 
The purpose of this paper is to give a concrete topological basis for the space of holomorphic functions on $\Omega$, denoted by $\mathcal{O}(\Omega)$.

\bigskip

There are several reasons to explore concrete topological basis of holomorphic functions on $\Omega$.
From geometric viewpoint, the complex manifold $\Omega$ is identified with the space of all the geodesic segments on $\Sigma$ with respect to the Poincar\'e metric: we can identify a point $[(z, w)] \in \D \times \D/\Gamma$ with a geodesic segment on $\Sigma$ given by
the projection of unique geodesic connecting $z$ and $w$ in the Poincar\'e disk $\D$.
It would be a natural question to describe holomorphic functions on the space of geodesic segments on $\Sigma$.

In view of the theory of automorphic forms, our goal is nothing but to describe the space of $\Gamma$-invariant holomorphic functions $\mathcal{O}(\D \times \D)^\Gamma$ on the bidisk $\D \times \D$. 
Although this would also be a natural question, automorphic forms on the bidisk 
with respect to actions by Fuchsian groups have not been well studied in the literature.

The complex manifold $\Omega$ is also a natural object in the study of holomorphic functions of several complex variables. 
We may regard $\Omega$ as a locally-trivial holomorphic $\D$-bundle over $\Sigma$ by its first projection $\pi|\Omega\colon \Omega \to \Sigma$, and embed $\Omega$ as a relatively compact domain in a holomorphic $\C\PP^1$-bundle over $\Sigma$ denoted by $X := \D \times \C\PP^1/\Gamma$ where $\Gamma$ acts on $\D \times \C\PP^1$ diagonally again. 
Then, we can see that the boundary $\bd\Omega$ of $\Omega$ in $X$, which is a real-analytic closed real hypersurface, has a special property, possessing a foliation by Riemann surfaces. 
Such a three-dimensional real hypersurface is called a {Levi-flat hypersurface}, and of great interest in connection with the Levi problem on complex manifolds and the problem on exceptional minimal sets of holomorphic foliations.
This $\bd\Omega$, which is diffeomorphic to the unit tangent bundle of $\Sigma$ and whose foliation agrees with the weakly stable foliation of the geodesic flow on $\Sigma$,  is known as a model example of Levi-flat hypersurface (see, for instance, \cite{brunella2010,diederich-ohsawa1985, ohsawa-book, ohsawa-kias}).

\bigskip

If we apply a general theory on holomorphic functions of several complex variables, 
we can easily confirm the existence of non-constant holomorphic functions on $\Omega$. 
Diederich and Ohsawa \cite{diederich-ohsawa1985} showed that $\Omega$ is 1-convex
by observing that 
\[
- \log \left( 1 - \left| \frac{w - z}{1 - \ol{z}w} \right| \right)
\]
is a plurisubharmonic exhaustion on $\Omega$ which is strictly plurisubharmonic on $\Omega \setminus D$,
where $D$ is the quotient of the diagonal set $\Delta \subset \D \times \D$ by $\Gamma$.
From the solution to the Levi problem by Grauert \cite{grauert}, the existence of such an exhaustion function
implies holomorphically convexity of $\Omega$, and moreover that any distinct points in $\Omega \setminus D$
are separated by holomorphic functions on $\Omega$.
Hence, $\Omega$ carries a plenty of holomorphic functions, in particular, those having singularities at points in the boundary $\bd\Omega$.

However, we have never seen any concrete holomorphic functions on $\Omega$ 
except for constant functions and those given by a Poincar\'e series
\[
\sum_{\gamma \in \Gamma} (\gamma(z) - \gamma(w))^{N}
\]
where $N \geq 2$ (Ohsawa \cite{ohsawa-kias}). 
Also, it has not been clear whether $\Omega$ admits a non-constant holomorphic function with slow growth, 
namely, those belonging to the weighted Bergman space $A^2_\alpha(\Omega)$
of weight order $\alpha > -1$  (for its definition, see \S\ref{Bergman space}).

\bigskip
The goal of this paper is to give a concrete description of $\mathcal{O}(\Omega)$, and to show the existence of 
holomorphic functions with slow growth:
\begin{MainTheorem}
We have an injective linear map
\[
I\colon \bigoplus_{N=0}^\infty H^0(\Sigma, K_\Sigma^{\otimes N}) \hookrightarrow \bigcap_{\alpha > -1} A^2_\alpha(\Omega) \subset \mathcal{O}(\Omega) 
\]
having dense image in $\mathcal{O}(\Omega)$ equipped with compact open topology and expressed by
\[
I(\psi)(z, w) =
\begin{cases}
\displaystyle \frac{1}{\Beta(N,N)} \int_{\tau \in \overline{zw}} \frac{\psi(\tau) (d\tau)^{\otimes N}}{[w, \tau, z]^{\otimes (N-1)}} & \text{for $N \geq 1$,} \\
\text{the constant $\psi$}  &\text{for $N = 0$}\\
\end{cases}
\]
as a function in $\mathcal{O}(\D \times \D)^\Gamma$ for $\psi = \psi(\tau) (d\tau)^{\otimes N} \in H^0(\Sigma, K_\Sigma^{\otimes N}) \subset H^0(\D, K_\D^{\otimes N})$ where $\tau$ is the coordinate of $\D$, the universal cover of $\Sigma$.
Here we denoted
\[
[w, \tau, z] := \frac{(w -z)  d\tau} {(w - \tau) (\tau - z)},
\]
an $\Aut(\D)$-invariant meromorphic 1-form in $\tau$ on $\D$, and $\Beta(p,q)$ is the beta function. 
\end{MainTheorem}

\begin{Corollary}
\label{non-vanishing}
The weighted Bergman space $A^2_\alpha(\Omega)$ is infinite dimensional for any $\alpha > -1$.
\end{Corollary}

Note that the canonical ring $R(\Sigma) := \bigoplus_{N=0}^\infty H^0(\Sigma, K_\Sigma^{\otimes N})$,
the domain of the operator $I$, is identified with 
the graded ring $\bigoplus_{N=0}^\infty H^0(D, \mathcal{I}_D^N/\mathcal{I}_D^{N+1})$ associated with
the filtered ring of jets of holomorphic functions along $D$, where $\mathcal{I}_D$ denotes the ideal sheaf of $D$. 
The main idea for the proof of Main Theorem is to extend jets of holomorphic functions belonging to $H^0(D, \mathcal{I}_D^N/\mathcal{I}_D^{N+1})$ to holomorphic functions on $\Omega$ so with minimal $L^2$ norm.
Our achievement is not only to construct such an optimal $L^2$ jet extension operator $I$ in our case, 
but also to obtain its explicit formula, and exact values of weighted $L^2$ norms of holomorphic functions extended by $I$ in terms of generalized hypergeometric functions.

Recently, general $L^2$ jet extension theorems of Ohsawa--Takegoshi type have been actively investigated (cf. \cite{popovici,hisamoto, demailly, hosono, cao-demailly-matsumura}).
However, none of them seems to be able to yield our optimal extension operator. 
We hope that our work would give some clue to pursue the general $L^2$ jet extension theorem with optimal constant.

\bigskip

Another possible approach to construct holomorphic functions on $\Omega$ with slow growth 
would be to use the $L^2$ extension theorem of Ohsawa--Takegoshi type. 
In \cite[Theorem 3.6]{guan-zhou} (cf. \cite{ohsawa-vi}), it was shown that, 
for any relatively compact domain $D$ with smooth strongly pseudoconvex boundary in a Stein manifold, 
there exists a continuous linear extension operator, which gives a gain in weight order, 
\[
I_H \colon A^2_{\alpha+1}(D \cap H) \to A^2_{\alpha}(D). 
\]
where $H$ is a closed smooth complex hypersurface in $\Omega$ and $\alpha \geq -1$.
This work suggests extending weighted $L^2$ holomorphic functions on a fiber of 
$\pi|\Omega\colon \Omega \to \Sigma$ to those on $\Omega$.
However, again, existing $L^2$ extension theorems seem not be able to produce holomorphic functions on $\Omega$ with slow growth as in our Main Theorem. 
The biggest difficulty in this approach is the lack of weighted $L^2$ estimate for the $\opa$-equation on weakly pseudoconvex domains in weight order $\alpha$ close to  $-1$. 
Although some useful Donnelly--Fefferman type estimates are available in weight order $\alpha > -\eta$, where $\eta$ is the Diederich--Forn{\ae}ss index of the pseudoconvex domain  (cf. \cite{berndtsson-charpentier, kohn1999, henkin-iordan, cao-shaw-wang, pinton-zampieri})), they are not satisfactory from our viewpoint
since it is known that the Diederich--Forn{\ae}ss indices of Levi-flat bounded domains are at most $1/2$ (cf. \cite{fu-shaw2014, adachi-brinkschulte2014}).
We do not know how to solve the $\opa$-equation on our $\Omega$, which is known to have the Diederich--Forn{\ae}ss index exactly $1/2$ (cf. \cite{adachi}), in weighted $L^2$ spaces of order $\alpha \in (-1, -1/2]$.

In this context, we would like to remark a work of Chen \cite{chen2014}, which motivated this work. He studied the weighted Bergman spaces of pseudoconvex domains in Euclidean spaces $\C^n$ with $C^2$ smooth boundary, and proved a H\"ormader type $L^2$ estimate for the $\opa$-equation
in weighted $L^2$ spaces of any weight order $\alpha > -1$ regardless of their Diederich--Fornaess indices. 
This affirmative result lead us to construct holomorphic functions on $\Omega$ with slow growth,
although the technical difficulty forced us to take a different approach from extending holomorphic functions on a fiber to those on $\Omega$.

\bigskip

As an application of our description of $\mathcal{O}(\Omega)$, we shall give an alternate proof for the following classical fact dated back to Hopf \cite{hopf} (cf. \cite{tsuji, sullivan, garnett, feres-zeghib}):
\begin{Corollary}
\label{cor:hopf}
The Hardy space $A^2_{-1}(\Omega)$ consists of constant functions only.
In particular, there is no bounded holomorphic function on $\Omega$ except constant functions. 
\end{Corollary}
Our new proof does not use any property of $\bd\Omega$
although previous proofs relied on the ergodicity of the Levi foliation on the boundary $\bd\Omega$.

As an another application, we shall give a Forelli--Rudin construction (cf. \cite{forelli-rudin, ligocka}) 
for weighted Bergman kernels of $\Omega$. 
\begin{Corollary}
\label{cor:ligocka}
For $\alpha > -1$, the weighted Bergman kernel $B_\alpha((z,w); (z',w'))$ of $A^2_\alpha(\Omega)$ (see \S\ref{Bergman space}
for the choice of measure and weight) has the following expression
\begin{align*}
&B_\alpha((z, w); (z', w')) \\
&= \frac{\Gamma(\alpha+2)}{\pi^2 (4g-4)} + \frac{1}{\pi} \sum_{N=1}^\infty \frac{1}{c_{N,\alpha}} \frac{1}{\Beta(N,N)^2}\int_{\tau \in \overline{zw}} \int_{\tau' \in \overline{z'w'}} \frac{B_N(\tau, \tau') (d\tau \otimes \overline{d\tau'})^{\otimes N}}{([w, \tau, z] \otimes \overline{[w', \tau', z']})^{\otimes (N-1)}}
\end{align*}
where  $g$ is the genus of $\Sigma$, $B_N(\tau,\tau') (d\tau \otimes \overline{d\tau'})^{\otimes N}$  is the Bergman kernel
of $K_\Sigma^{\otimes N}$, and
\[
c_{N,\alpha} := \frac{\Gamma(N+1)}{\Gamma(N+2+\alpha)} {}_3F_2 \left(\begin{matrix}
N+1, N, N \\
2N, N+2+\alpha
\end{matrix}
; 1
\right).
\]
Here $\Gamma(z)$ and ${}_3F_2(\cdots ; z)$ denote the gamma function and the generalized hypergeometric function respectively. 
\end{Corollary}

At the end, we shall explain the invariant holomorphic functions 
constructed by Ohsawa \cite{ohsawa-kias}
from our viewpoint. 
\begin{Corollary}
\label{cor:ohsawa}
For $N \geq 2$, we have 
\[
I(\sum_{\gamma \in \Gamma} \gamma^* d\tau^{\otimes N})(z,w)
 = \sum_{\gamma \in \Gamma} (\gamma(z) - \gamma(w))^{N}
\]
where $\sum_{\gamma \in \Gamma} \gamma^* d\tau^{\otimes N} \in H^0(\D, K_\D^{\otimes N})$ is regarded as a holomorphic $N$-differential on $\Sigma$. 
\end{Corollary}

It might be of interest that the extension operator $I$ enjoys such an algebraic identity although 
the operator $I$ is not a ring homomorphism from $R(\Sigma)$ to $\mathcal{O}(\Omega)$.  

\bigskip

This paper is organized as follows: In \S2, we give a quick review for $L^2$ theory of the $\opa$-equation on compact Riemann surfaces, and set up our notation.
In \S\ref{sect3}, we consider Taylor expansion of $f \in \mathcal{O}(\Omega)$ along the maximal compact analytic set $D$, and observe that its Taylor coefficients are identified with smooth sections of pluricanonical bundles of $\Sigma$
and they enjoy a system of $\opa$-equations on $\Sigma$.
When a jet of holomorphic function along $D$ is given, choosing it as an initial term and taking $L^2$ minimal solutions to this system of $\opa$-equations inductively, we obtain a formal power series along $D$. 
We explicitly compute the $L^2$ norms of these $L^2$ minimal solutions in \S\ref{sect4}, 
and show that the formal power series along $D$ actually defines a genuine holomorphic function on $\Omega$ in \S5.
In \S6, we complete the proof of Main Theorem. In particular, we prove the integral formula for the $L^2$ jet extension operator $I$.
In \S7, we explain Corollaries \ref{cor:hopf}, \ref{cor:ligocka} and \ref{cor:ohsawa}, some applications of our argument.

\bigskip
We will use the notation $f(x) \lesssim g(x)$ as $x \to \infty$ if there exists some $R > 0$ such that the ratio $f(x)/g(x)$ is bounded from above for $x > R$.
Also, we write $f(x) \approx g(x)$ as $n \to \infty$ when the ratio $f(x)/g(x) \to 1$ as $x \to \infty$.

\section{Preliminaries}

\subsection{$\opa$-equation and Green operator on the canonical ring of $\Sigma$}
We give a quick review for $L^2$ theory of the $\opa$-equation on compact Riemann surfaces to explain our notation and convention. 

Let $\Sigma$ be a compact Riemann surface, $g$ a K\"ahler metric of $\Sigma$ with the fundamental form $\omega_g$, 
and $(L, h) \to \Sigma$ a hermitian holomorphic line bundle. 
We consider $L^2$ Dolbeault complex 
\[
L^{(0,0)}_{(2)}(\Sigma, L) \overset{\opa}{\to} L^{(0,1)}_{(2)}(\Sigma, L),
\]
which is the completion of $C^\infty$ smooth Dolbeault complex
\[
C^{(0,0)}(\Sigma, L) \overset{\opa}{\to} C^{(0,1)}(\Sigma, L)
\]
with inner product
\[
\langle\langle u, v \rangle\rangle = \int_\Sigma \langle u, v \rangle_{h,g} \omega_g
\]
and whose $\opa$-operator is the unique closed extension of $\opa$ for smooth sections.
It is well known that their cohomology groups 
are isomorphic $ H^*_{(2)}(\Sigma, L) \simeq H^*(\Sigma, L)$.

Using local frame $e_\alpha$ of $L$, and local coordinate $z_\alpha$ of $\Sigma$, we write locally
\[
u = u_\alpha e_\alpha, \quad |u|^2_h = h_\alpha |u_\alpha|^2, \quad v = v_\alpha e_\alpha \otimes d\ol{z}_\alpha, \quad \omega_g = g_\alpha idz_\alpha \wedge d\ol{z}_\alpha, \quad |v|^2_{h, g} = \frac{h_\alpha|v_\alpha|^2}{g_\alpha}
\]
for $u \in L^{(0,0)}_{(2)}(\Sigma, L)$ and $v \in L^{(0,1)}_{(2)}(\Sigma, L)$.
Then, the adjoint operator $\opa^*_L$ of $\opa$, and the \emph{$\opa$-Laplacians} $\Box^{(0)}_L = \opa^*_L \opa$ and $\Box^{(1)}_L = \opa \opa^*_L$, which are essentially self-adjoint,  are expressed as 
\begin{align*}
\opa^*_L v =  \frac{-1}{g_\alpha}\left( \frac{\pa v_\alpha}{\pa z_\alpha} + \frac{\pa \log h_\alpha}{\pa z_\alpha} v_\alpha \right) e_\alpha
=  \frac{-1}{g_\alpha h_\alpha} \frac{\pa (h_\alpha v_\alpha)}{\pa z_\alpha}  e_\alpha,
\end{align*}
\begin{align*}
\Box^{(0)}_L u
 = \frac{-1}{g_\alpha}\left( \frac{\pa^2 u_\alpha}{\pa z_\alpha \pa\ol{z}_\alpha} + \frac{\pa \log h_\alpha}{\pa z_\alpha} \frac{\pa u_\alpha}{\pa \ol{z}_\alpha} \right) e_\alpha,
\end{align*}
and
\begin{align*}
\Box^{(1)}_L v 
 = \frac{\pa}{\pa \ol{z}_\alpha} \left( \frac{-1}{g_\alpha}\left( \frac{\pa v_\alpha}{\pa z_\alpha} + \frac{\pa \log h_\alpha}{\pa z_\alpha} v_\alpha \right) \right) e_\alpha \otimes d\ol{z}_\alpha
\end{align*}
for $u \in C^{(0,0)}(\Sigma, L)$ and $v \in C^{(0,1)}(\Sigma, L)$.

We denote by $G^{(1)}_L$ the \emph{Green operator} for $L$-valued $(0,1)$-forms, 
that is, $G^{(1)}_L$ is a bounded linear operator on $L^{(0,1)}_{(2)}(\Sigma, L)$ 
preserving  $C^{(0,1)}(\Sigma, L)$ and satisfying
\[
\Box^{(1)}_L G^{(1)}_L = G^{(1)}_L \Box^{(1)}_L   = I - H^{(1)}_L
\]
where $I$ denotes the identity map and $H^{(1)}_L\colon L^{(0,1)}_{(2)}(\Sigma, L) \to \Ker \Box^{(1)}_L$ is the orthogonal projector to the space of $\opa$-harmonic forms.
When the Dolbeault cohomology group $H^1(\Sigma, L) \simeq \Ker \Box^{(1)}_L$ vanishes, 
$\Box^{(1)}_L$ is invertible linear operator and $G^{(1)}_L$ is the inverse of $\Box^{(1)}_L$.
It follows that, for any $v \in L^{(0,1)}_{(2)}(\Sigma, L)$, $u := \opa^*_L G^{(1)}_L v \in L^{(0,0)}_{(2)}(\Sigma, L)$ solves the $\opa$-equation,
$\opa u = v$, and this $u$ has the smallest $L^2$ norm among all the solutions.

We assume that the genus of $\Sigma$ is $\geq 2$ in the rest of this paper, and fix its uniformization $\Sigma = \D / \Gamma$ where $\Gamma$ is a Fuchsian group.
We refer the natural coordinate $z$ of $\D \subset \C$ as the \emph{uniformizing coordinate}. 
We equip $\Sigma$ with the \emph{Poincar\'e metric} $g$ whose fundamental form is 
\[
\omega_g = \frac{2i dz \wedge d\ol{z}}{(1 - |z|^2)^2}
\]
on the uniformizing coordinate $z$. By abuse of notation, we denote its coefficient by
\[
g(z) := \frac{2}{(1 - |z|^2)^2}.
\]
We will work on $L = K_\Sigma^{\otimes n}$ where $K_\Sigma$ denotes the canonical bundle of $\Sigma$. 
The sections of $K_\Sigma^{\otimes  n}$ will be referred to as \emph{$n$-differentials} or \emph{$(n,0)$-differentials}.  
We also call $K_\Sigma^{\otimes  n}$-valued $(0,1)$-forms \emph{$(n,1)$-differentials}.
The Poincar\'e metric induces a hermitian metric on $K_\Sigma^{\otimes n}$. 
In the uniformizing coordinate $z$, our normalization is
\[
\left|(dz)^{\otimes n} \right|^2_{g} = \left( \frac{1 - |z|^2}{\sqrt{2}} \right)^{2n} = g^{-n}.
\]
The operators with respect to this metric will be denoted by 
$\opa^*_{K_\Sigma^{\otimes  n}}  =: \opa^*_n$, $\Box^{(0)}_{K_\Sigma^{\otimes  n}} =: \Box^{(0)}_n, \Box^{(1)}_{K_\Sigma^{\otimes  n}} =: \Box^{(1)}_n$
and $G^{(1)}_{K_\Sigma^{\otimes  n}} =: G^{(1)}_n$ for short.

Classical facts show the vanishing of $H^1(\Sigma, K_\Sigma^{\otimes  n})$ for $n \geq 1$,
hence, we have the following. 
\begin{Lemma}
\label{green}
Let $n \geq 1$ and $v \in C^{(0,1)}(\Sigma, K^{\otimes  n}_\Sigma)$. Then, $u := \opa^*_n G^{(1)}_n v \in C^{(0,0)}(\Sigma, K^{\otimes  n}_\Sigma)$ 
is the $L^2$ minimal solution to $\opa u = v$.
\end{Lemma}

\subsection{Weighted $L^2$ norms and Bergman spaces}
\label{Bergman space}
We denote the coordinate of the bidisk $\D \times \D$ in $\C^2$ by $(z, w)$. 
Throughout this paper, we will also use \emph{non-holomorphic} coordinate $(z, t)$ of the bidisk given by 
\[
\D \times \D \ni (z, w) \longmapsto (z, t = \frac{w - z}{1 - \ol{z}w}) \in \D \times \D,
\]
whose inverse transformation is
\[
\D \times \D \ni (z, t) \longmapsto (z, w = \frac{t + z}{1 + \ol{z}t}) \in \D \times \D.
\]
Note that  the diagonal set $\Delta \subset \D \times \D$ in $(z,w)$-coordinate is identified 
with a horizontal disk $\D \times \{ 0 \}$ in $(z,t)$-coordinate.

Let $\Sigma = \D/\Gamma$ be a compact Riemann surface of genus $\geq 2$ as above, and consider the quotient space $\Omega := \D \times \D /\Gamma$, 
where $\Gamma$ acts on $\D \times \D$ diagonally.  Let $X := \D \times \C\PP^1/\Gamma$ where $\Gamma$ acts on $\D \times \C\PP^1$ diagonally again. 
The first projection on $\D \times \C\PP^1$ induces a holomorphic submersion $\pi\colon X \to \Sigma$, which is a holomorphic $\C\PP^1$ bundle, and $\pi|\Omega$ is a holomorphic $\D$-bundle. 
Note that the quotient of the diagonal set $D := \Delta / \Gamma$ is the maximal compact analytic set in $\Omega$ biholomorphic to $\Sigma$.

We use a hermitian metric $G$ on $\Omega$ whose fundamental form is expressed as
\[
\omega_{G} := 
\frac{2i dz \wedge d\ol{z}}{(1 - |z|^2)^2} + \frac{(1 - |z|^2)^2}{| 1 - \ol{z}w|^4} \frac{i}{2} dw \wedge d\ol{w}.
\]
Note that $\omega_{G}$ agrees with $i dt \wedge d\ol{t}/2$ on each fiber of $\pi|\Omega$. 
We measure the $L^2$ norm of measurable functions on $\Omega$ with respect to its volume form
\[
dV := \frac{1}{2!} (\omega_G)^2 = \frac{4}{|1-\ol{z}w|^4} \frac{i}{2} dz \wedge d\ol{z} \wedge \frac{i}{2} dw \wedge d\ol{w}
\]
induced from $G$, and the weight function of the form $\delta^{\alpha}$ where 
\[
\delta := 1 - |t|^2 = 1 - \left|\frac{w - z}{1 - \ol{z}w}\right|^2 = \frac{(1 - |z|^2)(1 - |w|^2)}{|1 - \ol{z}w|^2}.
\]
We can easily see from direct computation that the metric $\omega_G$ and the weight function $\delta$ are invariant under the action of $\Gamma$, and well-defined on $\Omega$. 
Note also that $\Omega$ has finite volume with respect to $dV$. 

We let 
\begin{align*}
\langle\langle f, g \rangle\rangle^2_{\alpha} &:= \frac{1}{\Gamma(\alpha+1)} \int_{\Omega} f \overline{g} \delta^{\alpha} dV \\
&=  \frac{1}{\Gamma(\alpha+1)} \int_{\Omega} f(z, w)\overline{g(z,w)} \frac{4(1-|z|^2)^\alpha (1-|w|^2)^\alpha}{|1-\ol{z}w|^{4+2\alpha}} \frac{i}{2} dz \wedge d\ol{z} \wedge \frac{i}{2} dw \wedge d\ol{w}
\end{align*}
for a measurable function $f, g$ on $\Omega$ and $\alpha > -1$, and define the \emph{weighted $L^2$ space} by
\[
L^2_\alpha(\Omega) := \{ f: \text{measurable $\C$-valued function on $\Omega$} \mid \| f \|^2_{\alpha} := \langle\langle f, f \rangle\rangle_\alpha < \infty \}
\]
and the \emph{weighted Bergman space} by 
\[
A^2_\alpha(\Omega) := L^2_\alpha(\Omega) \cap \mathcal{O}(\Omega).
\]
It is well known that the weighted $L^2$ space is a separable Hilbert space,
and the weighted Bergman space is its closed subspace.

We will also use weighted $L^2$ Dolbeault complex on $\Omega$, 
\[
L^2_\alpha(\Omega) = L^{(0,0)}_{(2),\alpha}(\Omega) \overset{\opa}{\to} L^{(0,1)}_{(2),\alpha}(\Omega) \overset{\opa}{\to} L^{(0,2)}_{(2),\alpha}(\Omega),
\]
which is the completion of $C^\infty$ smooth Dolbeault complex with compact support
\[
C^{(0,0)}_c(\Omega) \overset{\opa}{\to} C^{(0,1)}_c(\Omega)  \overset{\opa}{\to} C^{(0,2)}_c(\Omega)
\]
with respect to inner product
\[
\langle\langle u, v \rangle\rangle_\alpha := \int_\Omega \langle u, v \rangle_{G} \delta^{\alpha }dV
\]
and whose $\opa$-operator is the maximal closed extension of $\opa$ for compactly supported $C^\infty$ smooth forms.

\section{A system of $\opa$-equations for differentials on $\Sigma$}
\label{sect3}
In this section, we consider Taylor expansion of $f \in \mathcal{O}(\Omega)$ along the maximal compact analytic set $D$, and derive a system of $\opa$-equations on $\Sigma$ which Taylor coefficients need to satisfy.

\subsection{$\opa$-equations for Taylor coefficients along the diagonal}
Let $f = f(z, w) \in \mathcal{O}(\D \times \D)$. 
We work on the coordinate $(z, t)$, and put $\widetilde{f}(z, t) := f(z, w(z, t))$.
Since $(z,t)$ is non-holomorphic coordinate, $\widetilde{f}$ is not holomorphic in $z$ in general.
Instead, $\widetilde{f}$ needs to satisfy the following equation:
\begin{align*}
0
& = \frac{\pa}{\pa \ol{z}}f(z, w) =  \frac{\pa}{\pa \ol{z}}f(z, w(z, t(z, w))) \\
& = \frac{\pa}{\pa \ol{z}}\widetilde{f}(z, t(z, w))= \frac{\pa}{\pa \ol{z}}\widetilde{f}(z, \frac{w - z}{1 - \ol{z}w})  \\
& =  \frac{\pa \widetilde{f}}{\pa \ol{z}} + \frac{\pa \widetilde{f}}{\pa t} \frac{w(w - z)}{(1 - \ol{z}w)^2}\\
& =  \frac{\pa \widetilde{f}}{\pa \ol{z}} +  \frac{t(t + z)}{1 - |z|^2}  \frac{\pa \widetilde{f}}{\pa t}.
\end{align*}

Let us denote Taylor coefficients of $f$ computed in the coordinate $(z, t)$ by 
\[
f_n(z) := \frac{1}{n!}\frac{\pa^n \widetilde{f}}{\pa t^n}(z, 0).
\]
Then, they enjoy
\begin{align*}
0
& =  \frac{1}{n!}\frac{\pa^n}{\pa t^n} \left(\frac{\pa \widetilde{f}}{\pa \ol{z}} +  \frac{t(t + z)}{1 - |z|^2}  \frac{\pa \widetilde{f}}{\pa t}\right) \Big|_{t = 0} \\
& =  \frac{\pa f_n}{\pa \ol{z}} + \frac{1}{(n-1)!} \frac{\pa^{n-1}}{\pa t^{n-1}}\frac{t + z}{1 - |z|^2}  \frac{\pa \widetilde{f}}{\pa t}\Big|_{t = 0} \\
& =  \frac{\pa f_n}{\pa \ol{z}} + \frac{nz}{1 - |z|^2}  f_n + \frac{n-1}{1 - |z|^2} f_{n-1}
\end{align*}
for $n \geq 1$.

\subsection{$\opa$-equations on $\Sigma$}
\label{equation}
Now we assume our $f \in \mathcal{O}(\D \times \D)$ is invariant under the action of $\Gamma$. Then, for each $\gamma = \begin{bmatrix} \alpha & \beta \\ \ol{\beta} & \ol{\alpha}\end{bmatrix} \in \Gamma < PSU(1,1)$, $|\alpha|^2 - |\beta|^2 = 1$, 
\begin{align*}
\widetilde{f}(z, t) & = f(z, w(z, t)) = f(\gamma z, \gamma w(z, t))\\
&= \widetilde{f}(\gamma z, t(\gamma z, \gamma w(z, t))) = \widetilde{f}(\gamma z, \frac{\beta \ol{z} + \alpha}{\ol{\beta} z + \ol{\alpha}}t)
\end{align*}
holds and the Taylor coefficients satisfy
\begin{align*}
f_n(z)
& =  \frac{1}{n!}\frac{\pa^n}{\pa t^n} \widetilde{f}(z, t) \Big|_{t = 0}
 =  \frac{1}{n!}\frac{\pa^n}{\pa t^n} \widetilde{f}(\gamma z, \frac{\beta \ol{z} + \alpha}{\ol{\beta} z + \ol{\alpha}}t) \Big|_{t = 0} \\
& =  \frac{1}{n!}\frac{\beta \ol{z} + \alpha}{\ol{\beta} z + \ol{\alpha}} \frac{\pa^{n-1}}{\pa t^{n-1}} \frac{\pa \widetilde{f}}{\pa t}(\gamma z, \frac{\beta \ol{z} + \alpha}{\ol{\beta} z + \ol{\alpha}}t) \Big|_{t = 0} \\
& =  \frac{1}{n!}\left( \frac{\beta \ol{z} + \alpha}{\ol{\beta} z + \ol{\alpha}} \right)^n\frac{\pa^{n} \widetilde{f} }{\pa t^{n}}(\gamma z, \frac{\beta \ol{z} + \alpha}{\ol{\beta} z + \ol{\alpha}}t) \Big|_{t = 0} 
 =   \left( \frac{\beta \ol{z} + \alpha}{\ol{\beta} z + \ol{\alpha}} \right)^n f_n(\gamma z). 
\end{align*}
Since $\gamma^* dz = {dz}/{(\ol{\beta} z + \ol{\alpha})^2}$, $\gamma^* d\ol{z} = d\ol{z}/(\beta \ol{z} + \alpha)^2$,
we have the following $\gamma$-invariance:
\begin{align*}
\gamma^*\left( f_n(z) \left(\frac{ \sqrt{2}dz }{1 - |z|^2}\right)^{\otimes  n} \right)  
& = f_n(\gamma z) \left( \frac{\sqrt{2} \gamma^* dz}{1 - |\gamma z|^2} \right)^{\otimes  n}\\
& =  \left(\frac{\ol{\beta} z + \ol{\alpha}}{\beta \ol{z} + \alpha} \right)^{n} f_n(z) \frac{ (\sqrt{2} dz)^{\otimes  n}}{(\ol{\beta} z + \ol{\alpha})^{2n} (1 - |\gamma z|^2)^{n}} \\
& =  \frac{f_n(z) (\sqrt{2} dz)^{\otimes  n}}{|\beta \ol{z} + \alpha|^{2n}}
\left(\frac{|\ol{\beta} z + \ol{\alpha}|^2}{|\ol{\beta} z + \ol{\alpha}|^2 - |\alpha z + \beta|^2}\right)^{n} \\
& =  f_n(z) \left(\frac{\sqrt{2} dz}{1 - |z|^2} \right)^{\otimes  n}.
\end{align*}
This leads us to define an $(n,0)$-differential 
\[
\varphi_n := f_n(z) \left(\frac{\sqrt{2}dz}{1-|z|^2}\right)^{\otimes  n}
\]
on $\Sigma$, which we call the $n$-th \emph{associated differential} of $f$. 
The associated differentials $\{ \varphi_n \}$ satisfy $\opa \varphi_0 = 0$ and
a system of $\opa$-equations:
\begin{align*}
\opa \varphi_n 
& =  \frac{\pa}{\pa \ol{z}} \left( f_n(z) \left(\frac{\sqrt{2}dz}{1-|z|^2}\right)^{\otimes n} \right) \otimes d\ol{z}
 = \left( \frac{\pa f_n}{\pa \ol{z}} + \frac{nz f_n}{1 - |z|^2} \right) \left(\frac{\sqrt{2}dz }{1 - |z|^2}\right)^{\otimes  n} \otimes d\ol{z}\\
& = -\frac{n-1}{1-|z|^2} f_{n-1} \left(\frac{\sqrt{2}dz }{1 - |z|^2}\right)^{\otimes n} \otimes d\ol{z}\\
& = -\frac{n-1}{\sqrt{2}} \varphi_{n-1} \otimes  \frac{2dz \otimes  d\ol{z}}{(1 - |z|^2)^2} 
 = -\frac{n-1}{\sqrt{2}} \varphi_{n-1}\otimes  \omega
\end{align*}
for $n \geq 1$. Here, we denoted a $\Gamma$-invariant $(1,1)$-differential by
\[
\omega := \frac{2dz \otimes d\ol{z}}{(1 - |z|^2)^2} = g(z) dz \otimes d\ol{z}.
\]

\section{$L^2$ estimate of the formal solution}
\label{sect4}
In this section, we shall determine the $L^2$ minimal solution to the system of $\opa$-equations on $\Sigma$, 
\[
\opa \varphi_0 = 0, \quad 
\opa \varphi_n = - \frac{n-1}{\sqrt{2}} \varphi_{n-1} \otimes \omega  \ (n \geq 1), 
\]
where $\varphi_n \in C^{(0,0)}(\Sigma, K_\Sigma^{\otimes n})$ and $\omega =2dz \otimes d\ol{z}/(1 - |z|^2)^2$.

\subsection{Multiplication by $\omega$}

We need some properties of the multiplication map
\[
a_{n-1}\colon C^{(0,0)}(\Sigma, K_\Sigma^{\otimes (n-1)}) \to C^{(0,1)}(\Sigma, K_\Sigma^{\otimes n}), \quad u \longmapsto  u \otimes \omega
\]
which appeared in the right hand sides of our $\opa$-equations.

\begin{Lemma}
\label{creation}
The map $a_{n-1}$ is norm-preserving linear isomorphism, and it holds that
\[
a_{n-1} \left( \Ker (\Box^{(0)}_{n-1} - \lambda I) \right) 
= \Ker (\Box^{(1)}_{n} - (\lambda + n - 1) I).
\]
\end{Lemma}

\begin{proof}
It is clear that $a_{n-1}$ is linear and bijective since $\omega$ is non-vanishing. 
Let us see that $a_{n-1}$ is norm-preserving. Take $u \in C^{(0,0)}(\Sigma, K^{\otimes (n-1)}_\Sigma)$, then, 
\begin{align*}
\| u \otimes \omega\|^2 
= \int_\Sigma | u \otimes \omega|^2_{g} \omega_g 
= \int_\Sigma |u |^2_{g} | \omega|^2_{g} \omega_g
= \int_\Sigma |u |^2_{g} \omega_g
= \| u \|^2. 
\end{align*}
Note that our normalization is
\[
| \omega|^2_{g} = \frac{4}{(1-|z|^2)^4} |dz|^2_{g} = 1.
\]

To show that $a_{n-1}$ maps an eigenform to another eigenform with greater eigenvalue, it suffices to prove the identity
\[
\Box^{(1)}_{n} \circ a_{n-1} = a_{n-1} \circ \Box^{(0)}_{n-1} + (n-1) a_{n-1}.
\]
Take $u \in C^{(0,0)}(\Sigma, K^{\otimes (n-1)}_\Sigma)$ and write $u = u(z) (dz)^{\otimes (n-1)}$ in the uniformizing coordinate, 
then,
\begin{align*}
\Box^{(1)}_{n} (u \otimes \omega )
& = \frac{\pa}{\pa \ol{z}} \left( \frac{-1}{g}\left( \frac{\pa (g u)}{\pa z} + \frac{\pa \log g^{-n}}{\pa z} g u \right) \right) (dz)^{\otimes n} \otimes d\ol{z} \\
& =  \frac{\pa}{\pa \ol{z}} \left( -\frac{\pa u}{\pa z} + (n-1) \frac{\pa \log g}{\pa z} u \right) (dz)^{\otimes n} \otimes d\ol{z} \\
& =  \left( -\frac{\pa^2 u}{\pa z \pa \ol{z}} + (n-1) \frac{\pa \log g}{\pa z} \frac{\pa u}{\pa \ol{z}} + (n-1) g u \right) (dz)^{\otimes n} \otimes d\ol{z} ,
\end{align*}
where our normalization is
\[
\frac{\pa \log g}{\pa z} = -2 \frac{\pa \log (1 - |z|^2)}{\pa z} = \frac{2 \ol{z}}{(1 - |z|^2)},
\]
\[
\frac{\pa^2 \log g}{\pa z \pa \ol{z}} = \frac{\pa}{\pa \ol{z}}  \frac{2 \ol{z}}{(1 - |z|^2)} = \frac{2}{(1 - |z|^2)^2} = g.
\]
On the other hand, in a coordinate trivializing $K^{\otimes(n-1)}_\Sigma$,
\begin{align*}
 \Box^{(0)}_{{n-1}}  u
&  = \frac{-1}{g} \left( \frac{\pa^2 u}{\pa z \pa\ol{z}} + \frac{\pa \log g^{-(n-1)}}{\pa z} \frac{\pa u}{\pa \ol{z}} \right) (dz)^{\otimes (n-1)}.
\end{align*}
We completed the proof.
\end{proof}

\subsection{$L^2$ minimal solution}
We shall solve the system of $\opa$-equation from a given holomorphic differential $\psi \in H^0(\Sigma, K^{\otimes N}_\Sigma)$
by taking $L^2$ minimal solutions inductively, and compute their $L^2$ norms precisely.
\begin{Proposition}
\label{estimate}
Let $\psi \in H^0(\Sigma, K^{\otimes N}_\Sigma)$ for some $N \geq 1$. Then, we have a solution $\{\varphi_n\} \in \prod_{n=0}^\infty C^{(0,0)}(\Sigma, K^{\otimes n}_{\Sigma})$, to 
\[
\opa \varphi_0 = 0, \quad
\opa \varphi_n = - \frac{n-1}{\sqrt{2}}  \varphi_{n-1}  \otimes \omega \ (n \geq 1)
\]
that satisfies $\varphi_n = 0 \ (n < N)$, $\varphi_N = \psi$ and
\[
\| \varphi_{N+m} \|^2 
= \frac{(2N-1)!}{\{(N-1)!\}^2}\frac{\{(N+m-1)!\}^2}{m! (2N+m-1)!} \| \psi \|^2
\]
for $m \geq 1$.
\end{Proposition}

\begin{proof}
We put $\varphi_n := 0 \ (n < N)$, $\varphi_N := \psi$, and
\[
\varphi_{N+m} := \opa^*_{N+m} G^{(1)}_{N+m} \left( - \frac{N+m-1}{\sqrt{2}} \varphi_{N+m-1} \otimes \omega \right)
\]
for $m \geq 1$ inductively. From Lemma \ref{green}, $\varphi_{N+m}$ gives the $L^2$ minimal solution to 
\[
\opa \varphi_{N+m} = -\frac{N+m-1}{\sqrt{2}} \varphi_{N+m-1} \otimes \omega.
\]
We shall show

\begin{Claim}
Each $\varphi_{N+m}$ is an eigenform of $\Box^{(0)}_{{N+m}}$ with eigenvalue $E_{N,m}$ where
\[
E_{N,m} := N + (N+1) + \dots + (N+m-1) = \frac{m(2N+ m-1)}{2}.
\]
\end{Claim}

Let us prove this claim by induction. The first case $m = 0$ is clear since $\varphi_N = \psi$ is holomorphic, hence, 
of eigenvalue $E_{N,0} = 0$. Now assume the case $m-1$ and show it for $m$. 
The assumption and Lemma \ref{creation} yield that $ \varphi_{N+m-1} \otimes  \omega$ is an eigenform of $\Box^{(1)}_{{N+m}}$ with eigenvalue
\[
E_{N,m-1} + (N+m-1) = E_{N,m}.
\]
Hence,
\[
G^{(1)}_{N+m} (\varphi_{N+m-1} \otimes \omega) = \frac{1}{E_{N,m}} \varphi_{N+m-1} \otimes  \omega.
\]
It follows that
\begin{align*}
\Box^{(0)}_{{N+m}} \varphi_{N+m} 
& = (\opa^*_{N+m} \opa ) \opa^*_{N+m} G^{(1)}_{N+m} \left( - \frac{N+m-1}{\sqrt{2}} \varphi_{N+m-1} \otimes \omega \right)\\
& = \opa^*_{N+m} \Box^{(1)}_{N+m} G^{(1)}_{N+m} \left( - \frac{N+m-1}{\sqrt{2}} \varphi_{N+m-1} \otimes \omega \right)\\
& = \opa^*_{N+m} \left( - \frac{N+m-1}{\sqrt{2}}  \varphi_{N+m-1}\otimes \omega \right)\\
& = \opa^*_{N+m} \left( - \frac{N+m-1}{\sqrt{2}} E_{N,m} G^{(1)}_{N+m}  \varphi_{N+m-1} \otimes \omega \right)\\
& = E_{N,m} \varphi_{N+m}.
\end{align*}
This shows the claim.

We continue the proof of Proposition \ref{estimate}. From the expression of $\varphi_{N+m}$, it follows that
\begin{align*}
\| {\varphi_{N+m}} \|^2 
& = \left(\frac{N + m - 1}{\sqrt{2}} \right)^2   \| \opa^*_{N+m} G^{(1)}_{N+m} (\varphi_{N + m - 1} \otimes \omega)  \|^2 \\
& = \frac{(N + m - 1)^2}{2}  \langle G^{(1)}_{N+m} (\varphi_{N + m - 1} \otimes \omega), \Box^{(1)}_{{N+m}} G^{(1)}_{N+m} (\varphi_{N + m - 1} \otimes \omega)  \rangle \\
& = \frac{(N + m - 1)^2}{2}  \langle \frac{1}{E_{N,m}}  \varphi_{N + m - 1} \otimes  \omega, \varphi_{N + m - 1} \otimes \omega \rangle \\
& = \frac{(N + m - 1)^2}{2E_{N,m}}  \| \varphi_{N + m - 1} \otimes \omega \|^2
 = \frac{(N + m - 1)^2}{2E_{N,m}}  \| \varphi_{N + m - 1} \|^2.
\end{align*}
This yields 
\begin{align*}
\| {\varphi_{N+m}} \|^2 
& = \frac{(N+m-1)^2}{2E_{N,m}}\frac{(N+m-2)^2}{2E_{N,m-1}}  \| \varphi_{N+m-2}\|^2\\
& = \left( \prod_{j=1}^m \frac{(N+m-j)^2}{2E_{N,m-j+1}} \right) \| \varphi_{N}\|^2\\
& = \left( \prod_{j=1}^m \frac{(N+m-j)^2}{(m-j+1)(2N+m-j)} \right) \|\psi \|^2\\
& = \frac{(2N-1)!}{\{(N-1)!\}^2} \frac{\{(N+m-1)!\}^2}{m! (2N+m-1)!} \| \psi \|^2
\end{align*}
inductively. The proof is completed.
\end{proof}

\section{Convergence of the formal solution}

In this section, we shall prove that the formal power series whose coefficients along $D$ 
are given by the solutions in Proposition \ref{estimate} 
actually define genuine holomorphic functions on $\Omega$.
Moreover, we shall show that they are weighted $L^2$ integrable for any weight order $\alpha > -1$, 
hence, belong to the weighted Bergman space $A^2_\alpha(\Omega)$.

\subsection{$L^2$ norms in terms of Taylor coefficients}
Let $f$ be a measurable function on $\Omega$, and we assume that $f$ is holomorphic along all the fibers of $\pi|\Omega\colon \Omega \to \Sigma$.
Then, in the same manner as in \S3.1, we can define the associated differentials $\{ \varphi_n \}$, which are measurable differentials on $\Sigma$, 
by expanding $f$ on each fiber of $\pi|\Omega$ at the intersection with the maximal compact analytic set $D$.
We shall give an expression for the weighted $L^2$ norm of $f$ in terms of $L^2$ norms of $\varphi_n$.

\begin{Lemma}
\label{norm}
Let $f$ be a measurable function on $\Omega$ which is holomorphic along all the fibers,
and denote by $\{ \varphi_n \}$ the associated differentials of $f$ on $\Sigma$. 
Then, we have 
\[
\| f \|^2_{\alpha} 
= \pi \sum_{n=0}^\infty \frac{\Gamma(n+1)}{\Gamma(n+2+\alpha)} \| \varphi_n \|^2.
\]
\end{Lemma}

\begin{proof}
Denote by $R \subset \D$ the fundamental domain of the universal covering map $\D \to \Sigma$. 
Then, the holomorphic disk bundle $\pi|\Omega\colon \Omega \to \Sigma$ is trivialized over $R$.
It is enough to compute the $L^2$ norm on $\Omega' := \pi^{-1}(R) \cap \Omega \simeq R \times \D$: 
\begin{align*}
\Gamma(\alpha+1) \| f \|^2_{\alpha} 
&=\int_{\Omega'} |f(z, w)|^2  \frac{4\delta^\alpha}{|1-\ol{z}w|^{4}} \frac{i}{2} dz \wedge d\ol{z} \wedge \frac{i}{2} dw \wedge d\ol{w}\\
&= \int_{R} 4d\lambda_z \int_{\D} |f(z, w)|^2  \frac{\delta^\alpha}{|1-\ol{z}w|^{4}} d\lambda_{w}\\
&= \int_{R} 4d\lambda_z \int_{\D} \left| \sum_{n=0}^\infty f_n(z) \left(\frac{w - z}{1 - \ol{z}w}\right)^n \right|^2  \frac{\delta^\alpha}{|1-\ol{z}w|^{4}} d\lambda_{w}
\end{align*}
where $\lambda$ denotes the two-dimensional Lebesgue measure on $\D$.
We change the coordinate of $\D$ from $w$ to $t = (w - z)/(1 - \ol{z}w)$.
Since $d\lambda_t =  (1 - |z|^2)^2 | 1 - \ol{z}w|^{-4} d\lambda_w$, 
we have
\begin{align*}
\Gamma(\alpha+1) \| f \|^2_{\alpha} 
&= \int_{R} \frac{4d\lambda_z}{(1-|z|^2)^{2}} \int_{\D} \left| \sum_{n=0}^\infty f_n(z) t^n \right|^2  (1-|t|^2)^\alpha d\lambda_{t}\\
&=\sum_{n=0}^\infty  \int_{R} |f_n(z)|^2 \frac{4d\lambda_z}{(1-|z|^2)^{2}} \int_{\D}  |t|^{2n}(1-|t|^2)^\alpha d\lambda_{t} \\
&=2\pi \sum_{n=0}^\infty  \int_{\Sigma} |\varphi_n|^2_g\,\omega_g \int_0^1  r^{2n+1}(1-r^2)^\alpha dr\\
&= \pi \Gamma(\alpha + 1) \sum_{n=0}^\infty  \|\varphi_n\|^2 \frac{\Gamma(n+1)}{\Gamma(n + 2 + \alpha)}.
\end{align*}
The proof is completed.
\end{proof}

\subsection{Convergence of the formal power series}

We denote the $L^2$ minimal solutions we obtained in Proposition \ref{estimate} by $\widehat{I}(\psi) := \{ \varphi_n \}_{n=0}^\infty$ for given $\psi \in H^0(\Sigma, K_\Sigma^{\otimes N})$, $N \geq 1$. 
For a constant function $\psi \in H^0(\Sigma, K_\Sigma^{\otimes 0})$, 
we let $\widehat{I}(\psi) := \{ \psi, 0, 0, \dots \}$.
Using $\widehat{I}(\psi)$, we consider formal power series
\[
f(z, w) := \sum_{n=0}^\infty f_n(z) t^n = \sum_{n=0}^\infty f_n(z) \left(\frac{w - z}{1 - \ol{z}w}\right)^n
\]
where 
\[
f_n(z) := \varphi_n \left(\frac{\sqrt{2}dz}{1-|z|^2}\right)^{-n}
\]
in the uniformizing coordinate. We shall show that $f$ converges and defines a genuine function in 
$L^2_\alpha(\Omega)$ for any $\alpha > -1$.

\begin{Proposition}
\label{convergence}
The sequence of partial sums $\left\{ F_n := \sum_{m=0}^n f_m(z) t^m \right\}_{n=0}^\infty$
is Cauchy in $L^2_\alpha(\Omega)$ for any $\alpha > -1$.
\end{Proposition}

\begin{proof}
From Lemma \ref{norm} and Proposition \ref{estimate}, we can compute, for any $n \geq 0$, 
\begin{align*}
\| F_{N+n} \|_\alpha^2 
& = \left\| \sum_{m=0}^{n} f_{N+m}(z) t^{N+m} \right\|_\alpha^2
= \pi  \sum_{m=0}^{n} \| \varphi_{N+m} \|^2 \frac{\Gamma(N+m+1)}{\Gamma(N+m+\alpha+2)} \\
& = \pi \| \psi \|^2 \sum_{m=0}^{n}  \frac{\Gamma(N+m+1)}{\Gamma(N+m+\alpha+2)}  \frac{(2N-1)!}{\{(N-1)!\}^2}\frac{\{(N+m-1)!\}^2}{(2N+m-1)!}\frac{1}{m!}.
\end{align*}
Hence, it is enough to show the convergence of the series 
\begin{align*}
 \sum_{m=0}^\infty \frac{\Gamma(N+m+1)}{\Gamma(N + m + 2 + \alpha)} 
\frac{(2N-1)!}{\{ (N-1)! \}^2}  \frac{\{(N+m-1)!\}^2}{(2N+m-1)!} \frac{1}{m!},
\end{align*}
which turns out to be a special value of the generalized hypergeometric function
\begin{align*}
&  \sum_{m=0}^\infty \frac{\Gamma(N+m+1)}{\Gamma(N + m + 2 + \alpha)} 
\frac{(2N-1)!}{\{ (N-1)! \}^2}  \frac{\{(N+m-1)!\}^2}{(2N+m-1)!} \frac{1}{m!}  \\
&= \frac{\Gamma(N+1)}{\Gamma(N+2+\alpha)}  \sum_{m=0}^\infty \frac{(N+1)_m}{(N + 2 + \alpha)_m} 
\frac{(N)_m (N)_m}{(2N)_m} \frac{1}{m!}   \\
&= \frac{\Gamma(N+1)}{\Gamma(N+2+\alpha)}  {}_3F_2 \left(\begin{matrix}
N+1, N, N \\
2N, N+2+\alpha
\end{matrix}
; 1
\right)
\end{align*}
where $(N)_m := N (N+1) \dots (N+m-1)$ and
\[
_pF_q\left(\begin{matrix}
a_1, \dots, a_p \\
b_1, \dots, b_q 
\end{matrix}; z
\right) 
:= \sum_{k=0}^\infty \frac{(a_1)_k \dots (a_p)_k}{(b_1)_k \dots (b_q)_k} \frac{z^k}{k!}
\]
is the generalized hypergeometric function. It is well-known that $_pF_q$ has finite value at $z=1$ if $\sum_{i=1}^p a_i < \sum_{j=1}^q b_j$.
In our case, this condition exactly corresponds to our assumption $\alpha > -1$.
\end{proof}

Next, we shall show that this $f$ is actually holomorphic. 
It is clear from Proposition \ref{convergence} that $f$ is holomorphic on each fiber of $\pi|\Omega$. 
However, since the partial sum $F_n$ may not be holomorphic in $z$ in general,
Proposition \ref{convergence} does not guarantee that $f$ is holomorphic in $z$.

\begin{Proposition}
\label{holomorphic}
The function $f$ constructed above is holomorphic on $\Omega$.
\end{Proposition}

\begin{proof}
From Proposition \ref{convergence}, $f$ belongs to $L^2_\alpha(\Omega)$ for any $\alpha > -1$.
We would like to show that $f$ belongs to the kernel of densely defined closed operator $\opa \colon L^2_\alpha(\Omega) \to L_{(2),\alpha}^{(0,1)}(\Omega)$ for some $\alpha > -1$.
Since $\Ker \opa$ is the orthogonal complement of the range of the adjoint operator $\opa^*$,
which is the minimum closed extension of the formal adjoint of $\opa$, 
it is enough to show $\langle\langle f, \opa^* \phi \rangle\rangle_\alpha = 0$ 
for any compactly supported smooth $(0,1)$-form $\phi \in C^{(0, 1)}_c(\Omega)$.
This orthogonality will follow from $\lim_{n \to \infty} \| \opa F_n \|_\alpha = 0$ since 
\begin{align*}
\left| \langle\langle f,  \opa^* \phi \rangle\rangle_{\alpha} \right| 
 =  \lim_{n \to \infty} \left| \langle\langle F_n, \opa^* \phi \rangle\rangle_{\alpha} \right| 
 =  \lim_{n \to \infty} \left| \langle\langle \opa F_n, \phi \rangle\rangle_{\alpha} \right| 
 \leq  \lim_{n \to \infty} \left\| \opa F_n \right \|_\alpha \left\| \phi  \right\|_\alpha = 0. 
\end{align*}

We are going to show $\lim_{n \to \infty} \| \opa F_n \|_\alpha = 0$ for some $\alpha$, which will be chosen later.
Thanks to the equation that $\{\varphi_n\}$ obeys, we have
\begin{align*}
\frac{\pa F_n}{\pa \ol{z}} 
&= \sum_{k=0}^n \left( \frac{\pa f_k}{\pa \ol{z}} \left(\frac{w - z}{1 - \ol{z}w}\right)^k 
+ \frac{k z f_k}{1 - |z|^2} \left(\frac{w - z}{1 - \ol{z}w}\right)^{k} 
+ \frac{k f_k}{1 - |z|^2} \left(\frac{w - z}{1 - \ol{z}w}\right)^{k+1} \right)\\
&= \frac{n f_n}{1 - |z|^2} \left( \frac{w - z}{1 - \ol{z}w} \right)^{n+1}.
\end{align*}
With the same notation as in the proof of Lemma \ref{norm}, we have
\begin{align*}
 \| \opa F_n \|^2_\alpha
&=  \left\| \frac{n f_n}{1 - |z|^2} \left( \frac{w - z}{1 - \ol{z}w} \right)^{n+1} d\ol{z} \right\|^2_\alpha \\
&= n^2 \int_{\Omega'} \frac{|f_n(z)|^2}{(1-|z|^2)^2}  \left| \frac{w - z}{1 - \ol{z}w} \right|^{2(n+1)} |d\ol{z}|^2_{g} \frac{4 \delta^\alpha}{|1 - \ol{z}w|^4}\frac{i}{2} dz \wedge d\ol{z} \wedge \frac{i}{2} dw \wedge d\ol{w}\\
&= \frac{n^2}{2} \int_{R} |f_n(z)|^2 \frac{i}{2} dz \wedge d\ol{z} \int_\D \left| \frac{w - z}{1 - \ol{z}w} \right|^{2(n+1)} \frac{4\delta^\alpha}{|1 - \ol{z}w|^4}  d\lambda_w.
\end{align*}
By changing the coordinate of $\D$ from $w$ to $t = (w - z)/(1 - \ol{z}w)$,
\begin{align*}
\left\| \opa F_n \right\|_\alpha^2
&= \frac{n^2}{2} \int_{R} |f_n(z)|^2 \frac{i}{2} dz \wedge d\ol{z} \int_\D \left| t \right|^{2(n+1)} \frac{4(1-|t|^2)^\alpha}{(1 - |z|^2)^2}  d\lambda_t \\
&= \frac{n^2}{2} \int_{R} | \varphi_n|^2_{g}\,\omega_g \int_\D \left| t \right|^{2(n+1)} (1-|t|^2)^\alpha  d\lambda_t \\
&= \pi n^2  \| \varphi_n \|^2  \int_0^1 r^{2(n+1)+1} (1-r^2)^\alpha  dr \\
&= \frac{\pi n^2}{2} \Gamma(\alpha+1) \| \varphi_n \|^2 \frac{\Gamma(n+2)}{\Gamma(n+3+\alpha)}.
\end{align*}
Since we know the growth of $\|\varphi_n\|^2$ from Proposition \ref{estimate}, we can estimate
\begin{align*}
\left\| \opa F_{N+m} \right\|_\alpha^2
& \lesssim (N+m)^2  \frac{\{(N+m-1)!\}^2\, (N+m+1)!}{m! (2N + m -1)! \, \Gamma(N+m+3+\alpha)}
\end{align*}
as $m \to \infty$. 
Stirling's formula, $\Gamma(n+1) \approx \sqrt{2\pi} n^{n+1/2} e^{-n}$ as $n \to \infty$,  yields
\begin{align*}
\left\| \opa F_{N+m} \right\|_\alpha^2
& \lesssim (N+m)^2  \frac{(N+m-1)^{2(N+m-1)+1}(N+m+1)^{N+m+1+1/2}}{m^{m+1/2} (2N + m -1)^{2N + m - 1/2} (N+m+2+\alpha)^{N+m+2+1/2+\alpha}}\\
& \lesssim m^{2 + 2(N+m-1)+1 + N+m+1+1/2 - (m+1/2) - (2N + m - 1/2) - (N+m+2+1/2+\alpha)} \\
&= m^{- \alpha}
\end{align*}
as $m \to \infty$. 
Hence, we have shown $\lim_{n \to \infty} \left\| \opa F_n \right\|_\alpha = 0$ when $\alpha > 0$. 
The proof is completed by choosing, for instance, $\alpha = 1$.
 \end{proof}

\section{Proof of Main Theorem}
\label{sect:proof}

Now we shall prove our Main Theorem. First we construct the desired extension operator 
\[
I\colon R(\Sigma)  \to \bigcap_{\alpha > -1} A^2_\alpha(\Omega) \subset \mathcal{O}(\Omega),
\]
where $R(\Sigma) := \bigoplus_{N=0}^\infty H^0(\Sigma, K_\Sigma^{\otimes N})$ is the canonical ring of $\Sigma$,
by summarizing our argument in previous sections. 

\begin{proof}[Construction of the extension operator] 
For constant functions in $H^0(\Sigma, K_\Sigma^{\otimes 0})$, we just map it to the same constant. 
For each $\psi \in H^0(\Sigma, K_\Sigma^{\otimes N})$, $N \geq 1$, 
Proposition \ref{estimate} yields the $L^2$ minimal solution $\widehat{I}(\psi) = \{ \varphi_n \}$ 
to the system of $\opa$-equation. 
We use $\{ \varphi_n \}$ as the coefficients of a formal power series along $D$, and define
\[
I(\psi)(z, w) := \sum_{n=0}^\infty f_n(z) t^n = \sum_{n=0}^\infty f_n(z) \left(\frac{w - z}{1 - \ol{z}w}\right)^n
\]
where 
\[
f_n(z) := \varphi_n \left(\frac{\sqrt{2}dz}{1-|z|^2}\right)^{-n}
\]
in the uniformizing coordinate $z$. Then, Proposition \ref{convergence} and \ref{holomorphic} guarantee 
that this formal power series $I(\psi)$ actually defines a holomorphic function belonging to the weighted Bergman space of weight order $> -1$.
We extend the map $I$ on $R(\Sigma)$ linearly, and obtain the extension operator $I$. 
\end{proof}

We will need later
\begin{Lemma}
\label{welldef_I}
The operator $I$ does not depend on the choice of uniformizing coordinate.
\end{Lemma}

\begin{proof}
Let $\psi \in H^0(\Sigma, K_\Sigma^{\otimes N})$, $N \geq 1$, and 
$\widehat{I}(\psi) = \{\varphi_n\}$. 
We take another uniformizing coordinate ${z'}$ of $\Sigma$. 
Namely, take $\gamma \in \Aut(\D)$ arbitrary 
and let ${z'} = \gamma z$, which express $\Sigma$ as $\D/\gamma \Gamma \gamma^{-1}$, ${z'} \in \D$. 
We shall show that $I(\psi)$ does not depend on these choices of uniformizing coordinate $z$ or ${z'}$. 

If we use ${z'}$ as the coordinate, $I(\psi)$ is given by a $\gamma \Gamma \gamma^{-1}$-invariant holomorphic
function on $({z'}, {w'}) \in \D \times \D$, 
\[
I'(\psi)(z', w') := \sum_{n=0}^\infty {f'_n}({z'}) \left(\frac{{w'} - {z'}}{1 - \ol{z'}w'}\right)^n
\]
where 
\[
{f'_n}(z') := \varphi_n \left(\frac{\sqrt{2}d{z'}}{1-|{z'}|^2}\right)^{-n}.
\]
We shall compute $I'(\psi)$ in $(z,w)$-coordinate. Since the identification between two coordinates is given by 
\[
(z',w') = (\gamma  z, \gamma w), \quad \gamma z = \frac{\alpha z + \beta}{\ol{\beta} z + \ol{\alpha}}
\]
where $|\alpha|^2 - |\beta|^2 = 1$, the coefficients transform as 
\[
f'_n(\gamma z) = f_n(z)  \left( \frac{{\beta}\ol{z} + \alpha}{\ol{\beta}z + \ol{\alpha}} \right)^{-n}
\]
by a computation similar to that in \S\ref{equation}, hence, we have
\[
I'(\psi)(\gamma z, \gamma w) = 
\sum_{n=0}^\infty f_n(z)  \left( \frac{{\beta}\ol{z} + \alpha}{\ol{\beta}z + \ol{\alpha}} \right)^{-n} \left(\frac{{\gamma w} - {\gamma z}}{1 - \ol{\gamma z}\gamma w}\right)^n = I(\psi)(z, w). 
\]
\end{proof}

Now we are going to prove our Main Theorem. It remains to check three points (Propositions \ref{injectivity}, \ref{density} and \ref{formula}).
\begin{Proposition}
\label{injectivity}
The operator $I$ is injective.
\end{Proposition}

\begin{proof}
The injectivity of $I$ on each summand of $\bigoplus_{n=0}^\infty H^0(\Sigma, K_\Sigma^{\otimes n})$ is clear
since for each $\psi \in H^0(\Sigma, K_\Sigma^{\otimes n})$ the $n$-th jet of $I(\psi)$, $[I(\psi)] \in H^0(D, \mathcal{I}_D^n/\mathcal{I}_D^{n+1})$
is identified with $\psi$ itself via the isomorphism $H^0(D, \mathcal{I}_D^n/\mathcal{I}_D^{n+1}) \simeq H^0(\Sigma, K_\Sigma^{\otimes n})$.
We shall show that $I(H^0(\Sigma, K_\Sigma^{\otimes n}))$ and $I(H^0(\Sigma, K_\Sigma^{\otimes m}))$ are orthogonal in $L^2_0(\Omega)$ if $n \neq m$, from which the injectivity of $I$ follows. 

Suppose $n_1 \neq n_2$ and take $\psi_1 \in H^0(\Sigma, K_\Sigma^{\otimes n_1})$ and $\psi_2 \in H^0(\Sigma, K_\Sigma^{\otimes n_2})$. 
Denote $\widehat{I}(\psi_1) = \{ \varphi_{1,n} \}$ and $\widehat{I}(\psi_2) = \{ \varphi_{2,n} \}$.
It follows from Claim in the proof of Proposition \ref{estimate} that $\varphi_{1,n}$ and $\varphi_{2,n}$ 
have different eigenvalue for $\Box_{n}^{(0)}$, hence, they are orthogonal in $L^2(\Sigma, K_\Sigma^{\otimes n})$. 
The computation in Lemma \ref{norm} therefore yields the orthogonality 
\[
\langle \langle I(\psi_1), I(\psi_2) \rangle \rangle_0 = \pi \sum_{n=0}^\infty \frac{\langle \langle \varphi_{1,n}, \varphi_{2,n} \rangle \rangle}{n+1} = 0.
\]
This completes the proof. 
\end{proof}

Next we would like to check 
\begin{Proposition}
\label{density}
The operator $I$ has dense image in $\mathcal{O}(\Omega)$ with respect to the compact open topology. 
\end{Proposition}

We need some additional notation. 
Consider $\Omega_\varepsilon := \{ [(z, w)] \in \Omega \mid \delta(z,w) > \varepsilon \}$ for $\varepsilon \in (0,1)$,
an exhaustion of $\Omega$, and denote by $A^2_0(\Omega_\varepsilon)$ the unweighted Bergman space on 
$\Omega_\varepsilon$ with respect to $dV$. 
From the computation in Lemma \ref{norm}, 
we know that the $L^2$ inner product of $A^2_0(\Omega_\varepsilon)$ by $\langle\langle \cdot, \cdot \rangle\rangle_{0,\varepsilon}$
is expressed as
\begin{align*}
\langle\langle f_1, f_2 \rangle\rangle_{0, \varepsilon} 
& = 2\pi \sum_{n=0}^\infty \langle\langle \varphi_{1,n}, \varphi_{2,n} \rangle\rangle \int_0^{\sqrt{1-\varepsilon}}  r^{2n+1} dr\\
& = \pi \sum_{n=0}^\infty \langle\langle \varphi_{1,n}, \varphi_{2,n} \rangle\rangle \frac{(1-\varepsilon)^{n+1}}{n+1}
\end{align*}
where  $\{ \varphi_{j,n} \}_n$ are the associated differentials of $f_j$ as in \S\ref{sect3}. 
We will also use
\[
\Pi^{(j)}_{n, \lambda}\colon L^{(0,j)}_{(2)}(\Sigma, K_\Sigma^{\otimes n}) \to  \Ker (\Box^{(j)}_{n} - \lambda I), \quad j = 0, 1, 
\]
the orthogonal projector to the $\lambda$-eigenspace of the $\opa$-Laplacian. 
Note that the orthogonal projector exists since the eigenspace is finite dimensional and forms a closed subspace thanks to the compactness of $\Sigma$.

\begin{proof}[Proof of Proposition \ref{density}]
It is enough to show that $I( R(\Sigma) )$ is dense in $A^2_0(\Omega_\varepsilon)$
since this density and Cauchy's estimate imply that $f \in \mathcal{O}(\Omega)$ is uniformly approximated on $\Omega_\varepsilon$ by functions belonging to $I( R(\Sigma) )$ for any $\varepsilon \in (0,1)$. 

Take $f \in A^2_0(\Omega_\varepsilon)$ which is orthogonal to $I(R(\Sigma))$. 
We shall show $f = 0$ by contradiction. 
Assume $f \neq 0$ at some point. 
Denote by $\{\varphi_n\}$ the associated differentials of $f$.
Since $I(R(\Sigma))$ contains constant functions, $f$ must be non-constant. 
Hence, there exists $N \in \N$ such that 
$\varphi_n = 0$ for $n < N$, and $\varphi_N \neq 0$.
This $\varphi_N$ must be a holomorphic $N$-differential since 
it coincides with the $N$-th jet of holomorphic function $f$.

We shall show $\widehat{I}(\varphi_N) = \{\Pi^{(0)}_{n, E_{N,n-N}} \varphi_n \}$ by induction. 
It holds for $n \leq N$. Assume it for $ n < N+m$ and consider the case $n = N+m$. 
Since $f - I(\varphi_N)$ is holomorphic, its associated differentials obey 
\[
\opa(\varphi_{N+m} - \varphi'_{N+m}) = -\frac{N+m-1}{\sqrt{2}} (\varphi_{N+m-1} - \varphi'_{N+m-1})\otimes  \omega
\]
where $\widehat{I}(\varphi_N) =: \{ \varphi'_n \}$. Note that $\varphi'_n \in \Ker(\Box^{(0)}_{n} - E_{N,n-N} I) $. By applying the projector $\Pi^{(1)}_{N+m, E_{N,m}}$, it follows from Lemma \ref{creation} and the assumption that 
\begin{align*}
& \opa(\Pi^{(0)}_{N+m, E_{N,m}} \varphi_{N+m} - \varphi'_{N+m}) \\
& = \frac{N+m-1}{\sqrt{2}} \Pi^{(0)}_{N+m-1, E_{N,m-1}}(\varphi_{N+m-1} - \varphi'_{N+m-1}) \otimes  \omega = 0
\end{align*}
Since $\Ker \opa \perp \Ker(\Box^{(0)}_{N+m} - E_{N,m} I)$, we conclude $\varphi'_{N+m} = \Pi^{(0)}_{N+m, E_{N,m}} \varphi_{N+m}$. 

The equality $\widehat{I}(\varphi_N) = \{\Pi^{(0)}_{n, E_{N,n-N}} \varphi_n \}$ yields a contradiction since
\begin{align*}
0 = \langle\langle f, I(\varphi_N) \rangle\rangle_{0, \varepsilon} 
& = \pi \sum_{m=0}^\infty \langle\langle \varphi_{N+m}, \Pi^{(0)}_{N+m, E_{N,m}} \varphi_{N+m} \rangle\rangle \frac{(1-\varepsilon)^{N+m+1}}{N+m+1}\\
& = \pi \sum_{m=0}^\infty \| \Pi^{(0)}_{N+m, E_{N,m}} \varphi_{N+m} \|^2  \frac{(1-\varepsilon)^{N+m+1}}{N+m+1}\\
& \geq \| \varphi_N \|^2  \frac{(1-\varepsilon)^{N+1}}{N+1} > 0.
\end{align*}
\end{proof}

The remaining thing to be shown is the expression of $I$ given in the statement of Main Theorem.
Let us see that the expression, which we temporarily denote by  $J(\psi)$, gives a $\Gamma$-invariant holomorphic function on $\D \times \D$. 

\begin{Lemma}
\label{welldef_J}
Let  $\psi \in H^0(\Sigma, K_\Sigma^{\otimes N})$, $N \geq 1$, and write $\psi = \psi(\tau) (d\tau)^{\otimes N}$ on the uniformizing coordinate $z$. 
Then,  
\[
J(\psi)(z, w) := \frac{1}{\Beta(N,N)} \int_z^w \left( \frac{(w - \tau) (\tau - z)}{w -z} \right)^{N-1} \psi(\tau) d\tau
\]
defines a $\Gamma$-invariant holomorphic function on $\D \times \D$, and it is independent of the choice of uniformizing coordinate. 
\end{Lemma}

\begin{proof}
Notice that 
\[
[w, \tau, z] := \frac{(w -z) d\tau}{(w - \tau) (\tau - z)} = \left( \frac{1}{w - \tau} + \frac{1}{\tau - z}\right) d\tau
\]
is a meromorphic 1-form in $(z, w, \tau) \in \D \times \D \times \D$ and invariant under the simultaneous action of $\Aut(\D)$ on $\D \times \D \times \D$
since it is a degenerate form of the cross ratio. Hence, when we fix distinct $z, w \in \D$, the integrand
\[
\frac{1}{\Beta(N,N)}\left(\frac{(w - \tau) (\tau - z)}{(w -z) d\tau}\right)^{\otimes (N-1)} \otimes \psi(\tau) (d\tau)^{\otimes N}
\]
is a holomorphic 1-form in $\tau$, and the value $J(\psi)(z,w)$ does not depend on the choice of integral path 
from $z$ to $w$. 
Hence, we see that $J(\psi)(z,w)$ is holomorphic in on $\D \times \D \setminus \Delta$ from its definition.

Along $\Delta$, $J(\psi)$ behaves as follows:
\begin{align*}
J(\psi)(z, w) 
& = \int_{z}^{z+(w - z)} \frac{1}{\Beta(N,N)}\left( \frac{(z + (w - z) - \tau) (\tau - z)}{(w - z)} \right)^{N-1} \psi(\tau) d\tau \\
& = \int_{0}^{1} \frac{1}{\Beta(N,N)}\left( \frac{((w - z) - (w - z) s) (w - z) s}{(w - z)} \right)^{N-1} \psi(z+(w - z) s) (w - z) ds \\
& = (w - z)^N \int_{0}^{1} \frac{1}{\Beta(N,N)}\left( (1- s) s \right)^{N-1} \psi(z+(w - z) s)  ds \\
& = (w - z)^N \int_{0}^{1}  \psi(z+(w - z) s)  \beta_N(ds)
\end{align*}
where we chosen the integral path as $\tau = z + s (w - z)$, $0 \leq s \leq 1$, and 
\[
\beta_N(ds) := \frac{s^{N-1}(1 - s)^{N-1} ds}{\Beta(N,N)} = \frac{(2N-1)!}{(N-1)!(N-1)!}{s^{N-1}(1 - s)^{N-1} ds}
\]
denotes the beta distribution with parameters $(N, N)$ on $s \in [0,1]$. 
Hence, $J(\psi)$ has zero of order $N$ along $\Delta$ being holomorphic on $\D \times \D$.

The invariance of $J(\psi)$ under $\Gamma$ follows from a computation: for each $\gamma \in \Gamma$,
\begin{align*}
&(\gamma^*J(\psi))(z,w) = J(\psi)(\gamma(z),\gamma(w)) \\
& = \int_{\gamma(z)}^{\gamma(w)} \frac{1}{\Beta(N,N)}\left( \frac{(\gamma(w) - \tau) (\tau - \gamma(z))}{(\gamma(w) -\gamma(z)) d\tau} \right)^{\otimes(N-1)} \otimes \psi(\tau) (d\tau)^{\otimes N} \\
& = \int_{z}^{w} \frac{1}{\Beta(N,N)}\left( \frac{(\gamma(w) - \gamma(\tau')) (\gamma(\tau') - \gamma(z))}{(\gamma(w) -\gamma(z)) d\gamma(\tau')} \right)^{\otimes(N-1)} \otimes \psi(\gamma(\tau')) (d\gamma(\tau'))^{\otimes N} \\
& = \int_{z}^{w} \frac{1}{\Beta(N,N)}\left( \frac{(w - \tau') (\tau' - z)}{(w -z) d\tau'} \right)^{\otimes(N-1)} \otimes \gamma^*\left(\psi(\tau) (d\tau)^{\otimes N}\right)\\
&= J(\gamma^*\psi)(z, w) = J(\psi)(z,w).
\end{align*}
During the computation, we used $\Gamma$-invariance of $\psi(\tau)(d\tau)^{\otimes N}$ 
and $\Aut(\D)$-invariance of $[w, \tau, z] $, 
 and we change the variable of the integral by $\tau = \gamma(\tau')$. 

We take another uniformizing coordinate ${z'}$ of $\Sigma$ given by ${z'} = \gamma z$ for arbitrary $\gamma \in \Aut(\D)$,
and denote by $J'(\psi)$ the path integral computed using $z'$-coordinate. 
Then, the given differential $\psi$ is expressed by pull-back $(\gamma^{-1})^* \psi \in H^0(\D, (K_\D)^{\otimes N})$ in $z'$-coordinate, 
hence, the same computation as in the proof for $\Gamma$-invariance of $J(\psi)$ yields 
\[
 J(\psi)(\gamma^{-1}z, \gamma^{-1}w) = J'(\psi)(z, w),
\]
which means that $J(\psi)$ does not depend on the choices of uniformizing coordinate $z$ or $z'$.
\end{proof}

\medskip

Now we are going to show that $J$ actually gives an expression for $I$. 

\begin{Proposition}
\label{formula}
For any $\psi \in H^0(\Sigma, K_\Sigma^{\otimes N})$, $N \geq 1$, $I(\psi) = J(\psi)$ holds. 
\end{Proposition}

\begin{proof}
It is enough to show it on $\{0\} \times \D$, namely, 
\[
I(\psi)(0, w) = J(\psi)(0,w) =  w^N \int_0^1  \psi(s w) \beta_N(ds)
\]
for each $w \in \D$ since our arguments do not depend on the choice of uniformizing coordinate $z$.
More precisely, to show
\[
I(\psi)(z_0, w_0) = J(\psi)(z_0,w_0)
\]
at $(z_0, w_0) \in \D \times \D$, we take $\gamma_0 \in \Aut(\D)$ such that 
$\gamma_0(0) = z_0$, and replace the original uniformization $\D/\Gamma$ 
by another uniformization $\D/\gamma_0^{-1}\Gamma\gamma_0$. 
Lemma \ref{welldef_I} and \ref{welldef_J} yield
\begin{align*}
I(\psi)(z_0, w_0) &= I'(\psi)(0, \gamma_0^{-1}(w_0)),\\
J(\psi)(z_0, w_0) &= J'(\psi)(0, \gamma_0^{-1}(w_0))
\end{align*}
respectively, where $I'$ and $J'$ denote the operators computed in the new uniformizing coordinate. 
Hence, showing $I(\gamma_0^*\psi)(0, \gamma_0^{-1}(w_0)) =  J(\gamma_0^*\psi)(0, \gamma_0^{-1}(w_0))$ is enough to conclude. 

\medskip

We shall write down $I(\psi)(0, w)$ explicitly. 
The $L^2$ minimal solution $\widehat{I}(\psi) =\{ \varphi_n \}$ satisfies
$\varphi_n = 0$ ($n < N$), $\varphi_N = \psi$, and
\begin{align*}
\varphi_{N+m} 
&= -\frac{N+m-1}{\sqrt{2} E_{N,m}}  \opa^*_{N+m}\left(\varphi_{N+m-1} \otimes \omega\right)\\
&= -\frac{\sqrt{2} (N+m-1)}{m(2N+m-1)}  \opa^*_{N+m}\left(\varphi_{N+m-1} \otimes \omega\right).
\end{align*}
We write 
\[
\varphi_{N+m-1} = f_{N+m-1} \left(\frac{\sqrt{2}dz}{1-|z|^2}\right)^{\otimes (N+m-1)} = f_{N+m-1} \sqrt{g}^{N+m-1} (dz)^{\otimes (N+m-1)}.
\]
Then, 
\begin{align*}
& \opa^*_{N+m}\left(\varphi_{N+m-1} \otimes \omega\right)\\
 &=  \frac{-1}{g g^{-(N+m)}} \frac{\pa (g^{-(N+m)} f_{N+m-1} \sqrt{g}^{N+m-1} g )}{\pa z}   (dz)^{\otimes (N+m)} \\
 &=  -g^{N+m-1} \frac{\pa (\sqrt{g}^{-(N+m-1)} f_{N+m-1} )}{\pa z}   (dz)^{\otimes (N+m)} \\
 &=  -\sqrt{g}^{N+m-2} \frac{\pa (\sqrt{g}^{-(N+m-1)} f_{N+m-1} )}{\pa z}   (\sqrt{g} dz)^{\otimes (N+m)}.
\end{align*}
Hence,
\begin{align*}
f_{N+m}(z)
&= \frac{\sqrt{2} (N+m-1)}{m(2N+m-1)}\sqrt{g}^{N+m-2} \frac{\pa (\sqrt{g}^{-(N+m-1)} f_{N+m-1} )}{\pa z}\\
&= \frac{\sqrt{2} (N+m-1)}{m(2N+m-1)} \frac{\sqrt{2} (N+m-2)}{(m-1)(2N+m-2)} \sqrt{g}^{N+m-2} \frac{\pa}{\pa z} \frac{1}{g} \frac{\pa (\sqrt{g}^{-(N+m-2)} f_{N+m-2} )}{\pa z} \\
&= \frac{(2N-1)!}{(N-1)!} \frac{\sqrt{2}^m (N+m-1)!}{m! (2N+m-1)!} \sqrt{g}^{N+m-2} \left(\frac{\pa}{\pa z} \frac{1}{g}\right)^{\circ (m-1)} \frac{\pa \sqrt{g}^{-N} f_{N}}{\pa z}.  
\end{align*}

Using explicit expression of $g$, 
\begin{align*}
& f_{N+m}(z) \\
&= \frac{(2N-1)!}{(N-1)!} \frac{(N+m-1)!}{m! (2N+m-1)!} \frac{1}{(1-|z|^2)^{N+m-2}} \left(\frac{\pa}{\pa z} {(1-|z|^2)^2} \right)^{\circ (m-1)} \frac{\pa}{\pa z} \left({1-|z|^2} \right)^{N} f_{N}
\end{align*}
follows, and at $z = 0$, 
\begin{align*}
f_{N+m}(0)
&= \frac{(2N-1)!}{(N-1)!} \frac{(N+m-1)!}{m! (2N+m-1)!} \frac{\pa^m f_{N}}{\pa z^{m}}(0).
\end{align*}
Note that we do not have lower order derivatives of $f_N$ in this expression 
since the $z$-derivative of $(1-|z|^2)$ vanishes at $z =0$.

We therefore showed that the holomorphic function $I(\psi)$ has the expansion 
\[
I(\psi)(0, w) = \frac{(2N-1)!}{(N-1)!} \sum_{m=0}^\infty \frac{(N+m-1)!}{(2N+m-1)! } \frac{1}{m!} \frac{\pa^m \psi}{\pa z^{m}}(0) w^{N+m}.
\]
Hence, it satisfies the differential equation
\begin{align*}
\frac{\pa^N}{\pa w^N} \left( w^{N-1} I(\psi)(0, w) \right)
& = \frac{(2N-1)!}{(N-1)!} \frac{\pa^N}{\pa w^N} \sum_{m=0}^\infty \frac{(N+m-1)!}{(2N+m-1)!} \frac{1}{m!} \frac{\pa^m \psi}{\pa z^{m}}(0) w^{2N+m-1}\\
& = \frac{(2N-1)!}{(N-1)!}  \sum_{m=0}^\infty \frac{1}{m!} \frac{\pa^m \psi}{\pa z^{m}}(0) w^{N+m-1}\\
& = \frac{(2N-1)!}{(N-1)!}  w^{N-1} \psi(w).
\end{align*}
Using an iterated integral on a path from $0$ to $w$, in particular, taking the path as the segment joining $0$ and $w$, we obtain
\begin{align*}
I(\psi)(0, w) 
& = \frac{(2N-1)!}{(N-1)!}  \frac{1}{w^{N-1}} \int_0^w d\sigma_{N} \dots \int_0^{\sigma_3} d\sigma_2 \int_0^{\sigma_2} \sigma_1^{N-1} \psi(\sigma_1) d\sigma_1 \\
& = \frac{(2N-1)!}{(N-1)!}  \frac{1}{w^{N-1}} \int_0^1 w ds_N \dots \int_0^{s_3} w ds_2 \int_0^{s_2} (s_1 w)^{N-1} \psi(w s_1) w ds_1 \\
& = \frac{(2N-1)!}{(N-1)!}  w^N \int_{\{ 0 < s_1 < \dots < s_N < 1\}} s_1^{N-1} \psi(s_1 w )  ds_1 ds_2 \dots ds_N \\
& = \frac{(2N-1)!}{(N-1)!}  w^N \int_0^1 ds_1 s_1^{N-1} \psi(s_1 w) \int_{s_1}^1 ds_2 \dots \int_{s_{N-1}}^1 ds_N \\
& = \frac{(2N-1)!}{(N-1)!}  w^N \int_0^1  \frac{s_1^{N-1}(1 - s_1)^{N-1} }{(N-1)!} \psi(s_1 w) ds_1\\
& =  w^N \int_0^1  \psi(s w) \beta_N(ds).
\end{align*}
We therefore have $I(\psi)(0, w) = J(\psi)(0, w)$, and completed the proof for Main Theorem. 
\end{proof}

\section{Applications}

We shall give two applications of our description of $\mathcal{O}(\Omega)$.
\subsection{The vanishing of Hardy space}
\label{Liouville}
As we have shown, the intersection of weighted Bergman spaces $\bigcap_{\alpha > -1} A^2_{\alpha}(\Omega)$ is infinite dimensional.
However, 

\setcounter{Corollary}{1}
\begin{Corollary}
The Hardy space $A^2_{-1}(\Omega)$ consists only of constant functions.
\end{Corollary}

The Hardy space of $\Omega$ is defined by
\[
A^2_{-1}(\Omega) := \{ f \in \mathcal{O}(\Omega) \mid \| f \|_{-1} < \infty \}.
\]
Here the Hardy norm $\| \cdot \|_{-1}$ is defined by
\[
\| f \|^2_{-1} 
:= \pi \sum_{n=0}^\infty \| \varphi_n \|^2.
\]
where $\{ \varphi_n\}$ are the associated differentials  of $f$.
We remark that our notation for Hardy space $A^2_{-1}(\Omega)$ 
is consistent with the notation for weighted Bergman spaces
$A^2_{\alpha}(\Omega)$ since Lemma \ref{norm} implies
\[
\| f \|_{-1}  = \lim_{\alpha \searrow -1} \|f \|_{\alpha}.
\]

\begin{proof}
Let $f \in \mathcal{O}(\Omega)$. We denote the associated differentials of $f$ by $\{\varphi_n\}$. 
If $f$ is non-constant, there exists $N \in \N$ such that $\varphi_N \neq 0$, and $\varphi_n = 0$ for $n < N$. 
From the argument in \S\ref{equation}, $\psi := \varphi_N$ is a holomorphic $N$-differential on $\Sigma$. 
Then, repeating the computation in the proof of Proposition \ref{holomorphic} with the orthogonal projectors 
\[
\Pi^{(0)}_{n, \lambda}\colon L^{(0,0)}_{(2)}(\Sigma, K_\Sigma^{\otimes n}) \to  \Ker (\Box^{(0)}_{n} - \lambda I),
\]
we have
\begin{align*}
\| \Pi^{(0)}_{N+m,E_{N,m}} \varphi_{N+m} \|^2
& = \frac{(N+m-1)^2}{2E_{N,m}} \| \Pi^{(0)}_{N+m-1, E_{N,m-1}} \varphi_{N+m-1} \|^2 \\
& = \frac{(2N-1)!}{\{(N-1)!\}^2} \frac{\{(N+m-1)!\}^2}{m! (2N+m-1)!} \| \Pi^{(0)}_{N, 0}\varphi_{N}\|^2
\end{align*}
for $m \geq 1$. Since $\| \Pi^{(0)}_{N, 0}\varphi_{N}\| = \| \varphi_{N}\| \neq 0$, 
it follows from the proof of Proposition \ref{convergence} that 
\begin{align*}
\| \Pi^{(0)}_{N+m,E_{N,m}} \varphi_{N+m} \|^2
& \approx C \frac{1}{m}
\end{align*}
as $m \to \infty$ for some constant $C > 0$. Hence,
\[
\| f \|^2_{-1} = \sum_{n=0}^\infty \| \varphi_n \|^2 
= \sum_{m=0}^\infty \| \varphi_{N+m} \|^2 
\geq \sum_{m=0}^\infty \|  \Pi^{(0)}_{N+m,E_{N,m}} \varphi_{N+m} \|^2 
\]
cannot converge.
\end{proof}

\subsection{Forelli--Rudin construction for $\Omega$}
Let $\alpha > -1$. 
We shall express the weighted Bergman kernel of $A^2_{\alpha}(\Omega)$
in terms of the Bergman kernels of $H^0(\Sigma, K_\Sigma^{\otimes N})$. 

\begin{Corollary}
For $\alpha > -1$, the weighted Bergman kernel $B_\alpha((z,w); (z',w'))$ of $A^2_\alpha(\Omega)$ has the following expression
\begin{align*}
& B_\alpha((z, w); (z', w')) \\
&=\frac{\Gamma(\alpha+2)}{\pi^2 (4g-4)} + \frac{1}{\pi} \sum_{N=1}^\infty \frac{1}{c_{N,\alpha}} \frac{1}{\Beta(N,N)^2} \int_{\tau \in \overline{zw}} \int_{\tau' \in \overline{z'w'}} \frac{B_N(\tau, \tau') (d\tau \otimes \overline{d\tau'})^{\otimes N}}{([w, \tau, z] \otimes \overline{[w', \tau', z']})^{\otimes (N-1)}}
\end{align*}
where  $B_N(\tau,\tau') (d\tau \otimes \overline{d\tau'})^{\otimes N}$  is the Bergman kernel
of $K_\Sigma^{\otimes N}$, 
\[
c_{N,\alpha} := \frac{\Gamma(N+1)}{\Gamma(N+2+\alpha)} {}_3F_2 \left(\begin{matrix}
N+1, N, N \\
2N, N+2+\alpha
\end{matrix}
; 1
\right),
\]
and $g$ is the genus of $\Sigma$. 
\end{Corollary}

\begin{proof}
We write down a complete orthonormal basis of $A^2_\alpha(\Omega)$ using the map $I$. 
Pick an orthonormal basis $\{ \psi_{N,j} = \psi_{N,j}(\tau)(d\tau)^{\otimes N} \}_{j=1}^{d_N}$ for each $H^0(\Sigma, K_\Sigma^{\otimes N})$, where $d_N := \dim H^0(\Sigma, K_\Sigma^{\otimes N})$. 
Note that $d_0 = 1$ and $\psi_{0,1} = 1/\sqrt{\mathrm{Vol}(\Sigma)} =  1/\sqrt{\pi (4g-4)}$ from the Gauss--Bonnet theorem. 
Recall that the Bergman kernel $B_N$ of $K_\Sigma^{\otimes N}$ is given by
\[
B_N = B_N(\tau, \tau')  (d\tau \otimes \overline{d\tau'})^{\otimes N} = \sum_{j=1}^{d_N} \psi_{N,j}(\tau) (d\tau)^{\otimes N} \otimes \overline{\psi_{N,j}(\tau') (d\tau')^{\otimes N}}.
\]

Now we collect all the holomorphic functions $\{ I(\psi_{N,j}) \}_{N,j}$. 
It is clear in view of our discussion in \S\ref{sect:proof} that
they form a complete orthogonal basis of $A^2_\alpha(\Omega)$. The computation in the proof of Proposition \ref{convergence} yields
\[
\| I(\psi) \|^2_\alpha 
= \pi \| \psi \|^2 \frac{\Gamma(N+1)}{\Gamma(N+2+\alpha)}  {}_3F_2 \left(\begin{matrix}
N+1, N, N \\
2N, N+2+\alpha
\end{matrix}
; 1
\right)
\]
for $\psi \in H^0(\Sigma, K_\Sigma^{\otimes N})$, hence, $\{ \sqrt{\Gamma(\alpha + 2)}/\pi\sqrt{4g-4} \} \cup \{ I(\psi_{N,j})/\sqrt{\pi c_{N,\alpha}} \}_{N \geq 1, j = 1, 2, \dots, d_N}$ is a complete orthonormal basis of  $A^2_\alpha(\Omega)$. We therefore obtain the expression for the Bergman kernel
\begin{align*}
&B_\alpha((z, w); (z', w')) - \frac{\Gamma(\alpha+2)}{\pi^2 (4g-4)}\\
&= \sum_{N=1}^\infty \sum_{j=1}^{d_N} \frac{1}{\pi c_{N,\alpha}}  I(\psi_{N,j})(z,w) \overline{I(\psi_{N,j})(z', w')}\\
&= \frac{1}{\pi} \sum_{N=1}^\infty \frac{1}{c_{N,\alpha}}  \sum_{j=1}^{d_N} \frac{1}{\Beta(N,N)^2}
\int_{\tau \in \overline{zw}} \frac{\psi_{N,j}(\tau) (d\tau)^{\otimes N}}{[w,\tau,z]^{\otimes (N-1)}} 
\overline{\int_{\tau' \in \overline{z'w'}} \frac{\psi_{N,j}(\tau') (d\tau)^{\otimes N}}{[w',\tau',z']^{\otimes (N-1)}}} \\
&= \frac{1}{\pi} \sum_{N=1}^\infty \frac{1}{c_{N,\alpha}}\frac{1}{\Beta(N,N)^2} \int_{\tau \in \overline{zw}} \int_{\tau' \in \overline{z'w'}} \frac{B_N(\tau, \tau') (d\tau \otimes \overline{d\tau'})^{\otimes N}}{([w, \tau, z] \otimes \overline{[w', \tau', z']})^{\otimes (N-1)}}.
\end{align*}
\end{proof}

\subsection{Invariant holomorphic functions given by Poincar\'e series}
We conclude this paper by expressing the invariant holomorphic functions constructed by Ohsawa \cite{ohsawa-kias} using our integral operator $I$.
Note that for $N \geq 2$, 
the Poincar\'e series $\sum_{\gamma \in \Gamma} \gamma^* d\tau^{\otimes N} \in H^0(\D, K_\D^{\otimes N})$ 
gives a holomorphic $N$-differential on $\Sigma$ 
and its convergence is uniform on each compact subset in $\D$. 

\begin{Corollary}
For $N \geq 2$, we have 
\[
I(\sum_{\gamma \in \Gamma} \gamma^* d\tau^{\otimes N})(z,w)
 = \sum_{\gamma \in \Gamma} (\gamma(z) - \gamma(w))^{N}.
\]
\end{Corollary}

\begin{proof} By direct computation. 
Using the invariance of $[w, \tau, z]$ under $\Aut(\D)$, we can compute 
\begin{align*}
I(\sum_{\gamma \in \Gamma} \gamma^* d\tau^{\otimes N})(z,w)
& =\frac{1}{\Beta(N,N)} \int_z^w \sum_{\gamma \in \Gamma} \frac{\gamma'(\tau)^N (w - \tau)^{N-1}(\tau - z)^{N-1}}{(w-z)^{N-1}} d\tau \\
& =\sum_{\gamma \in \Gamma}\frac{1}{\Beta(N,N)}  \int_{\gamma(z)}^{\gamma(w)} \frac{(\gamma(w) - \tau')^{N-1}(\tau' - \gamma(z))^{N-1}}{(\gamma(w)-\gamma(z))^{N-1}} d\tau',
\end{align*}
where we introduced new coordinate $\tau' = \gamma(\tau)$ on $\D$. 
We can compute each summand using the segment 
$\tau'(t) = \gamma(z) + t(\gamma(w)-\gamma(z))$, $t \in [0,1]$, as the integral path, 
\begin{align*}
& \frac{1}{\Beta(N,N)} \int_{\gamma(z)}^{\gamma(w)} \frac{(\gamma(w) - \tau')^{N-1}(\tau' - \gamma(z))^{N-1}}{(\gamma(w)-\gamma(z))^{N-1}} d\tau' \\
& = (\gamma(w)-\gamma(z))^N \int_{0}^1 \frac{(1-t)^{N-1} t^{N-1} dt}{\Beta(N,N)}\\ 
&  = (\gamma(w)-\gamma(z))^N.
\end{align*}
\end{proof}

\subsection*{Acknowledgement}
The author is grateful to Bo-Yong Chen for his inspiring lectures at Tongji University in 2012 and Nagoya University in 2013, which motivated this work. 
He would like to thank members of Saturday Seminar at Tokyo Institute of Technology for useful discussion, in particular, he is so indebted to Yoshihiko Mitsumatsu for the encouragement that convinced him to find the explicit description of $\mathcal{O}(\Omega)$. 
He is also grateful to Atsushi Yamamori for suggesting the expression of $c_{N,\alpha}$ of Corollary \ref{cor:ligocka} in terms of generalized hypergeometric functions, and to Takeo Ohsawa for providing \cite{ohsawa-kias} and encouraging him to study the invariant holomorphic function described in Corollary \ref{cor:ohsawa}.
It is his great pleasure to thank Judith Brinkschulte, Siqi Fu, Chin-Yu Hsiao, Kang-Tae Kim, Andrei Iordan, Xiaoshan Li, George Marinescu, Aeryeong Seo and Sungmin Yoo for useful conversation and discussion. 
Last but not least, he thanks the referees for their careful comments that improved the readability of this paper. 

Most part of this work was done during he was a member of Center for Geometry and its Applications, Pohang University of Science and Technology, and Tokyo University of Science. 
Some part of this work was done during his visit to Academia Sinica, Universit\"at zu K\"oln, Rutgers University in Camden, KIAS, University of Toledo and a conference dedicated to the memory of Giuseppe Zampieri at Levico Terme organized by CIRM. 
He is grateful to those institutes for their generous support. 
This work was partially supported by an NRF grant 2011-0030044 (SRC-GAIA) of the Ministry of Education, the Republic of Korea, and 
KAKENHI Grant Numbers 26800057 and JP18K13422 of Japan Society for the Promotion of Science.



\begin{bibdiv}
\begin{biblist}
\bib{adachi}{article}{
   author = {Adachi, Masanori},
   title = {A local expression of the Diederich--Fornaess exponent and the exponent of conformal harmonic measures},
   journal = {Bull. Braz. Math. Soc. (N.S.)},
   volume = {46},
   date = {2015},
   number = {1},
   pages = {65--79},
}
\bib{adachi-brinkschulte2014}{article}{
   author = {Adachi, Masanori},
   author = {Brinkschulte, Judith},
   title = {A global estimate for the Diederich--Fornaess index of weakly pseudoconvex domains},
   journal = {Nagoya Math. J.},
   volume = {220},
   date = {2015},
   pages={67--80},
}

\bib{berndtsson-charpentier}{article}{
   author={Berndtsson, Bo},
   author={Charpentier, Philippe},
   title={A Sobolev mapping property of the Bergman kernel},
   journal={Math. Z.},
   volume={235},
   date={2000},
   number={1},
   pages={1--10},
}


\bib{brunella2010}{article}{
   author={Brunella, Marco},
   title={Codimension one foliations on complex tori},
   journal={Ann. Fac. Sci. Toulouse Math.},
   volume={19},
   number={2},
   date={2010},
   pages={405--418},
}
		
\bib{cao-demailly-matsumura}{article}{
   author={Cao, JunYan},
   author={Demailly, Jean-Pierre},
   author={Matsumura, Shin-ichi},
   title={A general extension theorem for cohomology classes on non reduced
   analytic subspaces},
   journal={Sci. China Math.},
   volume={60},
   date={2017},
   number={6},
   pages={949--962},
}
\bib{cao-shaw-wang}{article}{
   author={Cao, Jianguo},
   author={Shaw, Mei-Chi},
   author={Wang, Lihe},
   title={Estimates for the $\overline\partial$-Neumann problem and
   nonexistence of $C^2$ Levi-flat hypersurfaces in $\mathbb{C}\mathrm{P}^n$},
   journal={Math. Z.},
   volume={248},
   date={2004},
   number={1},
   pages={183--221},
}

\bib{chen2014}{article}{
  author={Chen, Bo-Yong},
  title={Weighted Bergman spaces and the $\opa$-equation},
  journal = {Trans. Amer. Math. Soc.},
  volume={366},
  date={2014},
  number={8},
  pages={4127--4150}
}

\bib{demailly}{article}{
   author={Demailly, Jean-Pierre},
   title={Extension of holomorphic functions defined on non reduced analytic
   subvarieties},
   conference={
      title={The legacy of Bernhard Riemann after one hundred and fifty
      years. Vol. I},
   },
   book={
      series={Adv. Lect. Math. (ALM)},
      volume={35},
      publisher={Int. Press, Somerville, MA},
   },
   date={2016},
   pages={191--222},
}

\bib{diederich-ohsawa1985}{article}{
   author={Diederich, Klas},
   author={Ohsawa, Takeo},
   title={Harmonic mappings and disc bundles over compact K\"ahler
   manifolds},
   journal={Publ. Res. Inst. Math. Sci.},
   volume={21},
   date={1985},
   number={4},
   pages={819--833},
}

\bib{feres-zeghib}{article}{
   author={Feres, R.},
   author={Zeghib, A.},
   title={Leafwise holomorphic functions},
   journal={Proc. Amer. Math. Soc.},
   volume={131},
   date={2003},
   number={6},
   pages={1717--1725},
}

\bib{forelli-rudin}{article}{
   author={Forelli, Frank},
   author={Rudin, Walter},
   title={Projections on spaces of holomorphic functions in balls},
   journal={Indiana Univ. Math. J.},
   volume={24},
   date={1974/75},
   pages={593--602},
}

\bib{fu-shaw2014}{article}{
   author={Fu, Siqi},
   author={Shaw, Mei-Chi},
   title={The Diederich-Forn{\ae}ss exponent and non-existence of Stein domains with Levi-flat boundaries},
   journal={J. Geom. Anal.},
   volume={26},
   date={2016},
   number={1},
   pages={220--230},
}
\bib{garnett}{article}{
   author={Garnett, Lucy},
   title={Foliations, the ergodic theorem and Brownian motion},
   journal={J. Funct. Anal.},
   volume={51},
   date={1983},
   number={3},
   pages={285--311},
}
\bib{grauert}{article}{
   author={Grauert, Hans},
   title={On Levi's problem and the imbedding of real-analytic manifolds},
   journal={Ann. of Math. (2)},
   volume={68},
   date={1958},
   pages={460--472},
}
\bib{guan-zhou}{article}{
   author={Guan, Qi'an},
   author={Zhou, Xiangyu},
   title={A solution of an $L^2$ extension problem with an optimal
   estimate and applications},
   journal={Ann. of Math. (2)},
   volume={181},
   date={2015},
   number={3},
   pages={1139--1208},
}

\bib{henkin-iordan}{article}{
   author={Henkin, Gennadi M.},
   author={Iordan, Andrei},
   title={Regularity of $\overline\partial$ on pseudoconcave compacts and
   applications},
   journal={Asian J. Math.},
   volume={4},
   date={2000},
   number={4},
   pages={855--883},
}

\bib{hisamoto}{article}{
   author={Hisamoto, Tomoyuki},
   title={Remarks on $L^2$ jet extension and extension of singular hermitian metric with semipositive curvature},
   status={Preprint},
   eprint={arXiv:1205.1953}
}

\bib{hopf}{article}{
   author={Hopf, Eberhard},
   title={Fuchsian groups and ergodic theory},
   journal={Trans. Amer. Math. Soc.},
   volume={39},
   date={1936},
   number={2},
   pages={299--314},
}
\bib{hosono}{article}{
   author={Hosono, Genki},
   title={The optimal jet $L^2$ extension of Ohsawa-Takegoshi type},
   journal={Nagoya Math. J.},
   volume={239},
   date={2020},
   pages={153--172},
}

\bib{kohn1999}{article}{
   author={Kohn, J. J.},
   title={Quantitative estimates for global regularity},
   conference={
      title={Analysis and geometry in several complex variables},
      address={Katata},
      date={1997},
   },
   book={
      series={Trends Math.},
      publisher={Birkh\"auser Boston, Boston, MA},
   },
   date={1999},
   pages={97--128},
}

\bib{ligocka}{article}{
   author={Ligocka, Ewa},
   title={On the Forelli-Rudin construction and weighted Bergman
   projections},
   journal={Studia Math.},
   volume={94},
   date={1989},
   number={3},
   pages={257--272},
}

\bib{ohsawa-vi}{article}{
   author={Ohsawa, Takeo},
   title={On the extension of $L^2$ holomorphic functions. VI. A limiting
   case},
   conference={
      title={Explorations in complex and Riemannian geometry},
   },
   book={
      series={Contemp. Math.},
      volume={332},
      publisher={Amer. Math. Soc., Providence, RI},
   },
   date={2003},
   pages={235--239},
} 	
\bib{ohsawa-kias}{article}{
  author={Ohsawa, Takeo},
  title={Levi flat hypersurfaces --- Results and questions around basic examples},
  status={Manuscript distributed at ``Foliations in complex geometry and dynamics'' at KIAS, April 2016},
}
\bib{ohsawa-book}{book}{
   author={Ohsawa, Takeo},
   title={$L^2$ approaches in several complex variables},
   series={Springer Monographs in Mathematics},
   publisher={Springer, Tokyo},
   date={2018},
}
		

\bib{pinton-zampieri}{article}{
   author={Pinton, Stefano},
   author={Zampieri, Giuseppe},
   title={The Diederich-Fornaess index and the global regularity of the
   $\bar\partial$-Neumann problem},
   journal={Math. Z.},
   volume={276},
   date={2014},
   number={1-2},
   pages={93--113},
}

\bib{popovici}{article}{
   author={Popovici, Dan},
   title={$L^2$ extension for jets of holomorphic sections of a Hermitian
   line bundle},
   journal={Nagoya Math. J.},
   volume={180},
   date={2005},
   pages={1--34},
}
		
\bib{sullivan}{article}{
   author={Sullivan, Dennis},
   title={On the ergodic theory at infinity of an arbitrary discrete group
   of hyperbolic motions},
   conference={
      title={Riemann surfaces and related topics: Proceedings of the 1978
      Stony Brook Conference},
      address={State Univ. New York, Stony Brook, N.Y.},
      date={1978},
   },
   book={
      series={Ann. of Math. Stud.},
      volume={97},
      publisher={Princeton Univ. Press, Princeton, N.J.},
   },
   date={1981},
   pages={465--496},
}

\bib{tsuji}{book}{
  author={Tsuji, Masatsugu},
   title={Potential theory in modern function theory},
   publisher={Maruzen Co., Ltd., Tokyo},
   date={1959},
   pages={590},	
}

\end{biblist}
\end{bibdiv}
\end{document}